\documentclass{article}
\usepackage[utf8]{inputenc}

\usepackage{graphicx}
\usepackage{amsmath}
\usepackage{amsthm}
\usepackage{graphicx}
\usepackage[colorlinks=true, allcolors=blue]{hyperref}
\usepackage{color}
\usepackage{cite}
\usepackage{caption}
\usepackage{subcaption}
\usepackage{environ}

\newtheorem{theorem}{Theorem}
\newtheorem{lemma}{Lemma}

\newtheorem{definition}{Definition}
\newtheorem{problem}{Problem}
\newtheorem{observation}{Observation}

\newcommand{\oppos}{opposed}

\NewEnviron{resizealign}{\sbox0{\let\notag=\relax
    $\begin{matrix}\displaystyle\BODY\end{matrix}$}%
  \sbox1{$(\theequation)$}%
  \sbox2{\parbox{\dimexpr \wd0 + 2\wd1}%
    {\begin{flalign}\BODY\end{flalign}}}% for testing
  \noindent\resizebox{0.5\textwidth}{!}{\usebox2}%
}
\title{On star-$k$-PCGs: Exploring class boundaries for small $k$ values\thanks{A shorter version of this paper appeared as the conference article in \cite{Monti2023}.}}
\author{Angelo Monti\inst{1}\orcidID{0000-0002-3309-8249} \and
Blerina Sinaimeri\inst{2}\orcidID{0000-0002-9797-7592} }

\author{A. Monti\footnote{Computer Science Department, Sapienza University of Rome, Italy,
\href{mailto:monti@di.uniroma1.it}{monti@di.uniroma1.it}}
 \and B. Sinaimeri\footnote{Luiss University, Rome, Italy. \href{mailto:bsinaimeri@luiss.it}{bsinaimeri@luiss.it}}}
 
\begin{document}

\maketitle

\begin{abstract}
A graph $G$ is a star-$k$-PCG if there exists a  non-negative edge weighted star tree $S$ and $k$ mutually exclusive intervals $I_1, I_2, \ldots , I_k$ of non-negative reals such that each vertex of $G$ corresponds to a leaf of $S$ and there is an edge between two vertices in $G$ if the distance between their corresponding leaves in $S$ lies in $I_1\cup I_2\cup\ldots \cup I_k$.  These graphs are related to different well-studied classes of graphs such as PCGs and multithreshold graphs. It is well known that for any graph $G$ there exists a $k$ such  that $G$ is a  star-$k$-PCG. Thus, for a given graph  $G$ it is interesting to know which is the minimum $k$ such that $G$ is a  star-$k$-PCG.   

In this paper we focus on classes of graphs where $k$ is constant and prove that circular graphs and two dimensional grid graphs are  both star-$2$-PCGs and that they  are not star-$1$-PCGs.
Moreover we show that $4$-dimensional grids are at least star-$3$-PCG.\\
\noindent
\textit{Keywords:} Pairwise compatibility graph; Multithreshold graph; Graph theory; grid graphs.
\end{abstract}
\section{Introduction}\label{sec1}
The categorization of graphs into different classes is fundamental in graph theory and its applications, as it allows for a more structured and focused study of their properties and behaviors. Indeed, each class of graphs, possess unique characteristics that make them suitable for specific problems and applications. In essence, the diversity of graph classes reflects the diversity of real-world problems they are used to model and solve.  In this paper, we concentrate on a specific class of graphs, referred to as \emph{star-$k$-PCGs}. A graph $G=(V,E)$ is a star-$k$-PCG if there exists a weight function $w: V \rightarrow \mathcal{R}^+$ and $k$ mutually exclusive intervals $I_1, I_2, \ldots I_k$, such that there is an edge $uv \in E$ if and only if $w(u)+w(v) \in \bigcup_i I_i$.  This class serves as a bridge between two well-established graph classes: pairwise compatibility graphs and multithreshold graphs.  

\paragraph{Connection with pairwise compatibility graphs (PCGs).} A graph $G$ is a $k$-PCG (known also as multi-interval PCG) if there exists a  non-negative edge weighted  tree $T$ and $k$ mutually exclusive intervals $I_1, I_2, \ldots , I_k $ of non-negative reals such that each vertex of $G$ corresponds to a leaf of $T$ and there is an edge between two vertices in $G$ if the distance between their corresponding leaves in $T$ lies in $\bigcup_i I_i$ (see \textit{e.g.} \cite{ahmed17}). Such tree $T$ is called the \emph{$k$-witness} tree of $G$. The concept of $1$-PCGs, also known as PCGs, originated from the problem of reconstructing phylogenetic trees \cite{KPM03}.  Moreover, PCGs are a generalization of the well-known $k$-\emph{leaf power} graphs \cite{NISHIMURA2002} and have proven valuable in describing and analyzing evolutionary processes \cite{Long2020}.
 
One of the most important open problems in the field is, whether given an integer $k$, the $k$-PCG can be recognized in polynomial time, and it is unknown whether this problem can be solved in polynomial time, even for the case of $k=1$. To make progress towards the solution of this problem, restrictions have been made on the topology of the $k$-witness tree into two main directions: a star and a caterpillar (see \emph{e.g.}\cite{Brandstdt2008,Calamoneri2013_threshold,Telle2022}). The class of star-$k$-PCGs is exactly the class of $k$-PCGs for which the witness tree is a star \cite{Monti2023,Monti23_ictcs2023}.  Figure~\ref{fig:example} depicts an example of a graph that is a star-$1$-PCG.  This topology constraint on the witness tree, has been proven valuable as  for star witness trees, the decision problem becomes simpler compared to the general case. Indeed, Xiao and Nagamochi \cite{Xiao2020} introduced the first polynomial-time algorithm for identifying graphs that are star-$1$-PCGs. Next,  Kobayashi \emph{et al.} in \cite{Kobayashi22} improved upon this result by introducing a new characterization of star-$1$-PCGs that led to a linear time recognition algorithm. 

\paragraph{Connection with multithreshold graphs.}
Multithreshold graphs were introduced by Jamison and Sprague \cite{Jamison2021} in 2020 as a generalization of the class of threshold graphs introduced by Chvátal and Hammer \cite{Chvtal1977} in 1977 and has since become one of the most prominent and well‐studied graph classes (see \cite{Mahadev_threshold}).  In a similar way, multithreshold graphs have gained considerable interest within the research community since their introduction, as evidenced by the following studies \cite{Jamison2021,Puleo2020,chen2022}.  Given real numbers $\theta_1, \theta_2, \ldots, \theta_k$, with $\theta_1 < \theta_2< \ldots < \theta_k$ we say that a  graph $G=(V,E)$ is a $k$-threshold graph with thresholds $\theta_1, \theta_2, \ldots, \theta_k$ if there exist  an assignment $r : V \rightarrow \mathcal{R}$
of real ranks to the vertices such that for every pair of distinct vertices
$u,v \in V$ we have $uv \in E $ if and only if the inequality $\theta_i \leq r(u) + r(v)$ holds for
an odd number of indices $i$.  %As observed in \cite{Puleo2020} if we set $\theta_{k+1}=+\infty$ then the constraint  we want is the following:  $uv \in E$ if and only if  $r(v) + r(w) \in [\theta_{2i-1}, \theta_{2i}]$ for some $i$. 
It was shown in \cite{Kobayashi22} that $2$-threshold graphs are exactly star-$1$-PCGs. Here we directly extend this claim by showing in Observation~\ref{prop:2k-threshold} that the class of star-$k$-PCGs is equivalent to the class of $2k$-threshold graphs.

\begin{figure}
\centering
    \begin{subfigure}[t]{0.25\textwidth}
    \centering
    \includegraphics[width=\linewidth]{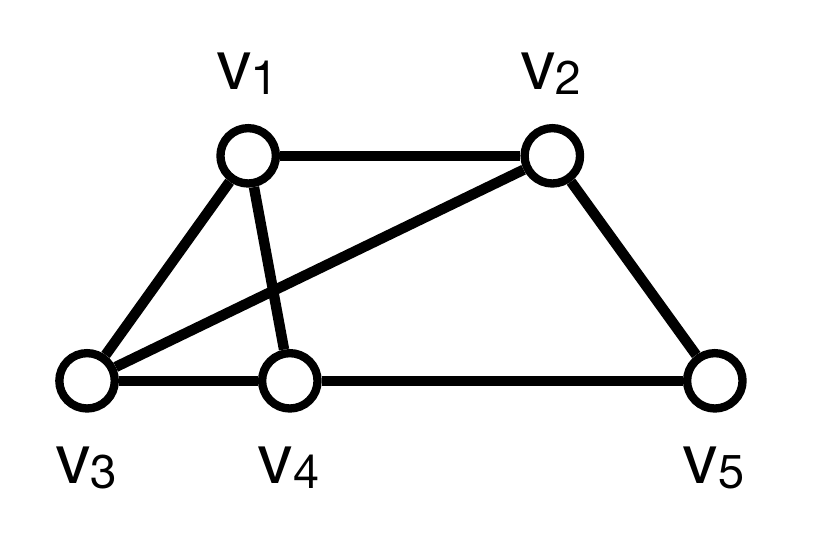}
    \caption{\label{fig:image1a}}
    \end{subfigure}
    \begin{subfigure}[t]{0.3\textwidth}
    \centering
    \includegraphics[width=\linewidth]{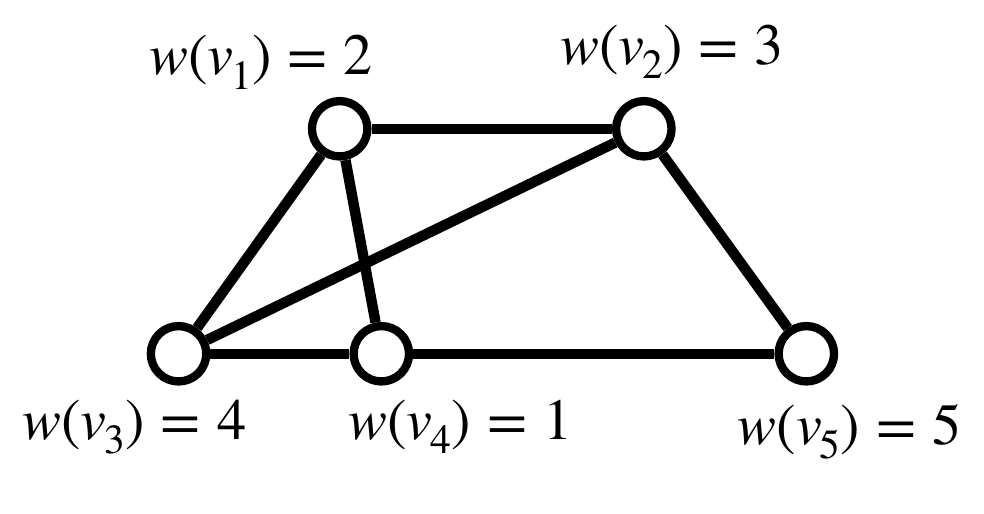}
    \caption{\label{fig:image1b}}
    %\vspace{0.2cm}
    \end{subfigure}
    \begin{subfigure}[t]{0.25\textwidth}
    \centering
    \includegraphics[width=\linewidth]{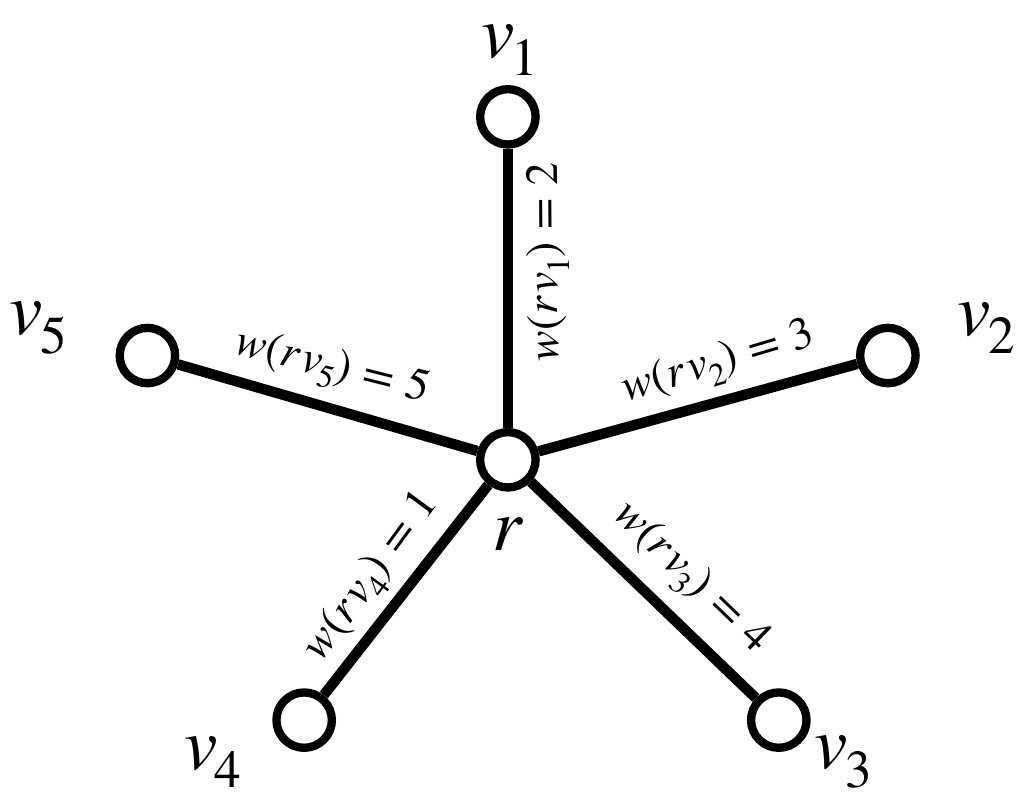}
    \caption{\label{fig:image1c}}
    \end{subfigure}
    \caption{(\subref{fig:image1a}) An example of a graph $G$ that is a star-$1$-PCG graph.  In (\subref{fig:image1b}) and (\subref{fig:image1c}) the $1$-witness graph and the $1$-witness star for which $G$ is a star-$1$-PCG with $I_1=[5,8]$.} 
\label{fig:example}
\end{figure}

\paragraph{Results}
It is already established that every graph $G$ is a star-$k$-PCG for some positive integer $k$ \cite{ahmed17}.  Hence, we introduce the following notation. 
\medskip
\begin{definition}\label{def:star_number}
Given a graph $G$, we define \emph{the star number}, $\gamma(G)$, to be the smallest positive integer $k$, such that $G$ is a star-$k$-PCG.
\end{definition}

\medskip
\noindent
The star number is not know even for graphs belonging to simple classes. In this paper we firstly focus on $n$-vertex graphs for small values of $n$.  Identifying the smallest graphs that are excluded from a class not only helps to define the boundaries of that class but could also help towards a characterization of the graph class through forbidden subgraphs. In this framework we consider the following question:  \emph{what is the smallest value of $n$ for which there exists an $n$-vertex graph that is not a star-$k$-PCG}? This question has been already investigated for related graphs classes. Indeed, it is known that the smallest graphs that are not  $1$-PCGs have 8 vertices \cite{Calamoneri2012,Durocher2015} while the smallest graphs that are not  $2$-PCGs must have at least 9 vertices \cite{Calamoneri2023}.  Here we determine the exact value of $\gamma(G)$ for all the graphs $G$ with at most 7 vertices. By doing so we show that the smallest graphs with star number 2 are only 4 and have exactly 5 vertices; the smallest graphs with star number 3 are only 3 and have exactly 7 vertices.

Next we study some simple graph classes: caterpillars, cycle graphs, grid graphs. Caterpillars have already been shown to be $2$-multithreshold \cite{Jamison20} and thus they are also star-$1$-PCGs. Here we provide a different construction in the context of star-$k$-PCGs. While path graphs are known to be star-$1$-PCGs \cite{Xiao2020}, cycle graphs are only shown to be $1$-PCGs of caterpillar witness trees \cite{Yanhaona09}. However, when considering a star tree structure, we prove that the star number of cycle graphs is 2. In \cite{hakim2022} it was proved that $2$-dimensional grid graphs are $1$-PCGs of a caterpillar. Here we prove that the star number of two dimensional grid graphs $G_{n_1,n_2}$,  is 1 if $\min\{n_1,n_2\}\leq 2$ and 2 otherwise. From the results in \cite{Jamison20} it can be easily shown that the star number $d$-dimensional grids is at least $d-3$. Here we improve this result for $d=4$, by showing that the star number of $4$-dimensional grids is at least $3$. All our constructions can be obtained in linear time.

%%%%%%%%%%%%%%%%%%%%%%%%%%%%%%%%%%%%%%%%%%%%%%%%%%%
\section{Preliminaries}\label{sec:Preliminaries}
In this paper we  only consider simple graphs, that is graphs that contain no loops or multiple edges. Additionally, we focus only undirected graphs and thus for simplicity, we use a notational shorthand, writing $uv$ to represent the unordered pair  $\{u,v\}$.  For a graph $G=(V,E)$ and a vertex $u\in V$, the set $N(u)=\{v: uv\in E\}$ is called the \emph{neighborhood} of $u$. 
%needed for the grids:
A \emph{cycle} graph, denoted as $C_n$, $n\geq 3$,  is a graph that consists of a single cycle of $n$ vertices.  A \emph{caterpillar} is a tree in which all the vertices are within distance 1 of a central path. The central path contains only vertices of degree at least 2 (\textit{i.e.} vertices that are not leaves).

For any integer $n\geq 1$ we denote by $[n]$ the set $\{0,1,2,\ldots,n-1\}$.  A \emph{$d$-dimensional grid} graph $G_{n_1,\ldots,n_d}$, is a graph such that the vertex set is given by \mbox{$[n_1]\times [n_2] \times \ldots \times [n_d]$} and there is an edge between two vertices if and only if they differ in exactly one coordinate and the difference is 1. More formally, a vertex $u$ is described by its coordinates $(i_1,\ldots, i_d)$. For any dimension $j$ we denote by $u_j$ the coordinate of $u$ in the $j$-th dimension.  Two vertices $u$ and $u'$ are adjacent if there is a dimension $i$ such that $|u_{i} - u'_{i}| = 1$ and for all $l\neq i$, $u_{l}=u'_{l}$ (see Figure~\ref{fig:gr4x2} for an example).

Given a graph $G=(V,E)$ and a function $w : V \rightarrow \mathcal{R}^+$, we denote by $G^w$ its weighted version, that is, $G$ supplemented with the weight function $w$ (see Figure.~\ref{fig:image1a}-~\ref{fig:image1b}). Moreover for each vertex $v\in V$,  $w(v)$ denotes the weight of the vertex $v$ and for each pair of vertices $u$ and $v$  of $G$ we denote by $w(uv)$ the sum $w(u)+w(v)$.

Given $G^w$ we define $\sigma(G^w)=v_1,v_2\ldots,v_{|V|}$
as the sequence of vertices in $G^w$, arranged in non-descending order based on their weights. 
 Given a family $\mathcal{S}$, of $r$ pairwise disjoint subsets of $V$  we will use the notation $S_1, S_2, \ldots, S_r$ where $S_i \in \mathcal{S}$, to denote the partial order where for all $1\leq i \leq r$ the weights of the vertices in $S_i$ are all less or equal to the weights of all vertices in $S_{i+1}$. 

We say that the neighborhood $N(u)$ is \emph{consecutive} with respect to $\sigma(G^w)$ if for $v_i,v_j \in N(u)$ with $i< j$ implies $v_t \in N(u)$ for any $i<t<j$.

\medskip
\begin{definition}\label{def:star-k}
A graph $G=(V,E)$ is a star-$k$-PCG if there exists a weight function $w: V \rightarrow \mathcal{R}^+$ and $k$ mutually exclusive intervals $I_1, I_2, \ldots I_k$, such that there is an edge $uv \in E$ if and only if $w(u)+w(v) \in \bigcup_i I_i$. Such weighted graph $G^w$  is called the \emph{$k$-witness}  of $G$.
\end{definition}

The following result is a straightforward generalization of Observation 2.2 in \cite{Kobayashi22}.

\medskip
\begin{observation}\label{prop:2k-threshold}
Given an integer $k$, a graph is a $2k$-threshold graph if and only if it is a star-$k$-PCG.
\end{observation}
\begin{proof}
($\leftarrow$)  Let $G = (V, E)$ be a $2k$-threshold graph defined by $r: V \rightarrow \mathcal{R}$  and  $\theta_1 < \theta_2< \ldots < \theta_k$  such that for every pair of distinct vertices $u,v \in V$ we have $uv \in E$ if and only if $\theta_i \leq r(u) + r(v)$ for an odd number of indices $i$.   Notice that \emph{w.l.o.g.} we can assume that $\theta_i \geq 0$ for all $1\leq i \leq 2k$ and that $r(v) > 0$ for all $v\in V$.  To show that $G$ is a $k$-star-PCG we construct $G^w$ as follows.  For every $v\in V$ we define $w(v)=r(v)$ and  $I_i= [\theta_{2i-1}, \theta_{2i})$ with $1\leq i \leq k$. It is easy to see that if we want to have closed intervals as in the definition of star-$k$-PCGs, it is sufficient to set $I'_i= [\theta_{2i-1}, \theta_{2i}']$ where $\theta_{2i}'=\max\{w(e): e \in E \text{ and } w(e) \in I_i\}$. Finally, by construction for all $e=uv$ it holds that $w(e) \in I_i$ if and only if $\theta_{2i-1} \leq r(u)+r(v) < \theta_{2i}$, meaning that the inequality is valid for an odd number of indices.

($\rightarrow$) Now, let $G$ be a star-$k$-PCG and let  $G^w$ be its $k$-witness for $I_i= [a_i, b_i]$ with $1\leq i \leq k$.  Notice that we can substitute $I_i=[a_i, b_i]$ with $I_i=[a_i, b'_i)$ as follows. For $1\leq i <k$ we set $b'_i=\frac{a_{i+1}+b_i}{2}$ if there is no edge $e$ for which $b_i < w(e) < a_{i+1}$, otherwise  we set $b'_i=\min \{w(e): e \not\in E \text{ and } b_i < w(e) < a_{i+1}\}$. If $i=k$ we set $b'_k=b_k+1$  if there is no edge $e$ for which $b_k < w(e)$, otherwise we set  $b'_k=\min \{w(e): e \not\in E \text{ and } b_k < w(e)\}$. It is easy to see that $G^w$ still remains a $k$-witness. 
Now \emph{w.l.o.g.} we can assume $a_1 < b_1 < \ldots < a_k < b_k$. Then $G$ is clearly a $2k$-threshold with $r=w$ and $1\leq j \leq 2k$  $\theta_j=a_{(j+1)/2}$ if $j$ is odd or $\theta_j=b_{j/2}$ otherwise.  The proof follows straightforwardly by the construction.
\end{proof}

\noindent
Given a graph $G=(V,E)$, we say that $G'=(V', E')$ is an \emph{induced subgraph} of $G$ if $V' \subset V$ and $E'$ consists of all the edges in $E$ whose endpoints are both in $V'$. The following lemma can be directly deduced from Definition~\ref{def:star_number}.
\medskip
\begin{lemma}\label{lem:remove_vertices}
Given a graph $G$ it holds that $\gamma(G') \leq \gamma(G)$ for any induced subgraph $G'$ of $G$. 
\end{lemma}

\noindent
The next lemma follows trivially by the definition of star-$k$-PCG.
\medskip
\begin{lemma}\label{lem:twins}
Let $G$ be a star-$k$-PCG and let $G^w$ be a $k$-witness for $G$. For any two vertices $u,v$ in $G^w$ for which $N(u)-\{v\} \neq N(v) - \{u\}$  it holds $w(u)\neq w(v)$.
\end{lemma}

The following definition will be needed throughout the paper.
\medskip
\begin{definition}\label{def:kFP}
Let $G^w$ be a vertex weighted graph and $k$ an integer such that $k \geq 1$. Consider a sequence, $c_1, c_2, \ldots, c_{2k+1}$ of $2k+1$  pairs of vertices in  $G^w$. This sequence is said a $k$-FP if it satisfies the following conditions:
\begin{enumerate}
    \item A pair $c_i$ is an edge in $G^w$ (i.e., $c_i \in E$) if and only if $i$ is an odd integer.
    \item The weights of the pairs follow an ascending order: $w(c_1) \leq  w(c_2) \leq  \ldots \leq  w(c_{2k+1})$.
\end{enumerate}
\end{definition}

%\textcolor{red}{non sose e' troppo ma  forse per aiutare il lettore a capire il lemma che segue si potrebbe dire che ovviamente i $k+1$ archi $c_i$ con $i$ dispari denotano la presenza di almeno $k+1$ intervalli.}
\noindent
See Figure~\ref{fig:image2} for an example of a $3$-FP.  The next lemma shows that the $k$-FP is a forbidden pattern (hence the name) for any graph that is  a star-$k$-PCG.
\medskip
\begin{lemma} \label{lem:kFP}
Let $G$ be a star-$k$-PCG than any  $k$-witness  $G^w$ of $G$  does not contain a $k$-FP.
\end{lemma}

\begin{proof}
Firstly, observe that an equality within the sequence of inequalities in item 2 of Definition~\ref{def:kFP} implies the existence of an edge and a non-edge with identical weight which  contradicts the definition of a star-$k$-PCG. So we have $c_i \neq c_j$ for all $i\neq j$. The proof follows by noticing that if $G$ contains a $k$-FP then no two edges $c_i$ and $c_j$ can belong to the same interval. Given that there are exactly $k+1$ edges  in the $k$-FP sequence, then at least $k+1$ distinct intervals are required.
\end{proof}

\noindent
We introduce now the concept of $k$-FSeq.

\medskip
\begin{definition}\label{def:kPath}
Let $G^w$ be a vertex weighted graph and $k$ an integer such that $k \geq 1$. Consider a sequence of $k+2$ vertices in  $G^w$, denoted as $u_1,u_2,\ldots u_{k+2}$. This sequence is said a \emph{$k$-FSeq} if it satisfies the following conditions:
\begin{enumerate}
\item  $u_iu_{i+1} \in E$ for  $1\leq i\leq k+1$ 
\item  $u_iu_{i+2}\not\in E$ for  $1\leq i \leq k$ 
\item $w(u_1)<w(u_2)<\ldots <w(u_{k+2})$ 
\end{enumerate}
\end{definition}

\noindent
See Figure~\ref{fig:image1} for an example of a $3$-FSeq. The following lemma holds.

\begin{lemma} \label{lem:Fseq}
Let $G$ be a star-$k$-PCG than any  $k$-witness  $G^w$ of $G$  does not contain a $k$-FSeq.
\end{lemma}

\begin{proof} Notice that if $G^w$ contains a $k$-FSeq
then the edges  $$u_1 u_2, u_1 u_3, u_2 u_3, u_2 u_4, \ldots , u_{k} u_{k+1}, u_{k} u_{k+2}, u_{k+1} u_{k+2}$$ form a $k$-FP (see  Fig.~\ref{fig:image1} and Fig.~\ref{fig:image2} for an example). The proof follows trivially by Lemma~\ref{lem:kFP}.
\end{proof}

Note that from the proof of Lemma~\ref{lem:Fseq}, we can infer that a $k$-FSeq implies the existence of a $k$-FP. However, the converse is not always true. This is shown in Fig.~\ref{fig:image3}, where we present a $3$-FP edge sequence that does not lead to the direct identification of a $3$-FSeq.

\begin{figure}
\centering
    \begin{subfigure}[t]{0.35\textwidth}
    \centering
    \includegraphics[width=\linewidth]{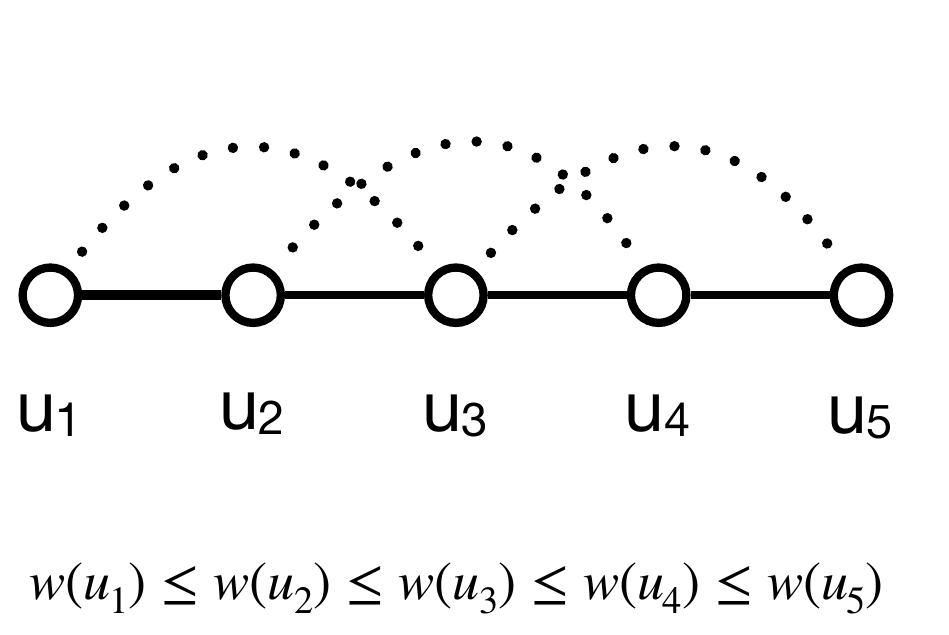}
    \caption{\label{fig:image1}}
    \end{subfigure}
    \vspace{0.1cm}
    \begin{subfigure}[t]{0.44\textwidth}
    \centering
    \includegraphics[width=\linewidth]{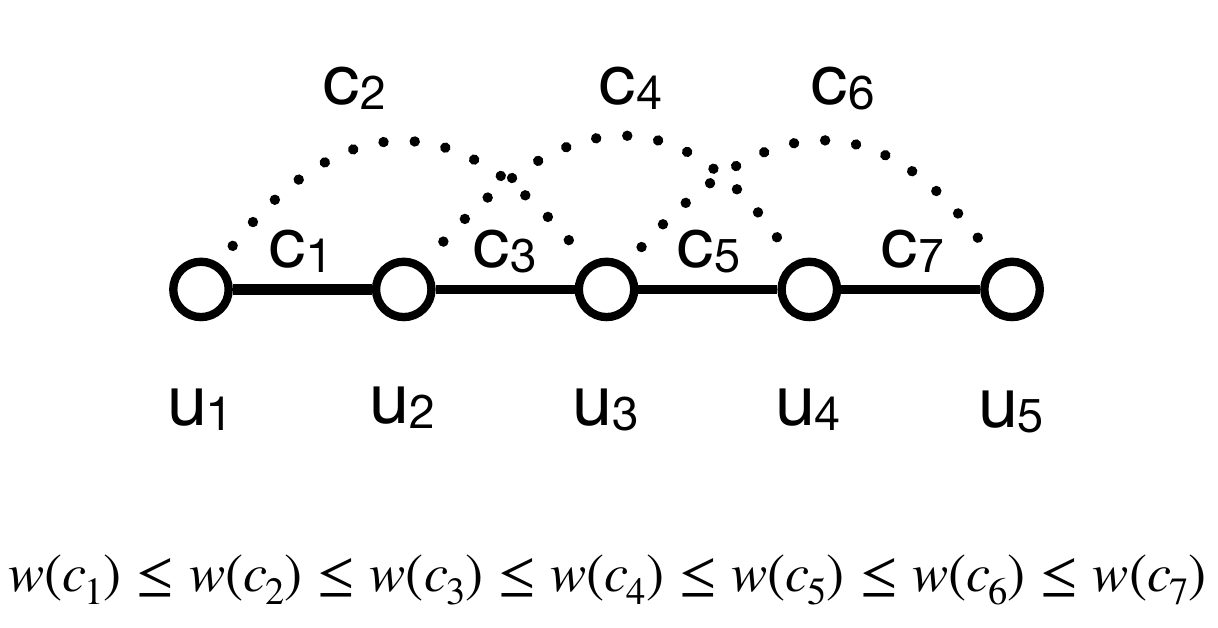}
    \caption{\label{fig:image2}}
    \end{subfigure}
    \begin{subfigure}[t]{0.4\textwidth}
    \centering
    \includegraphics[width=\linewidth]{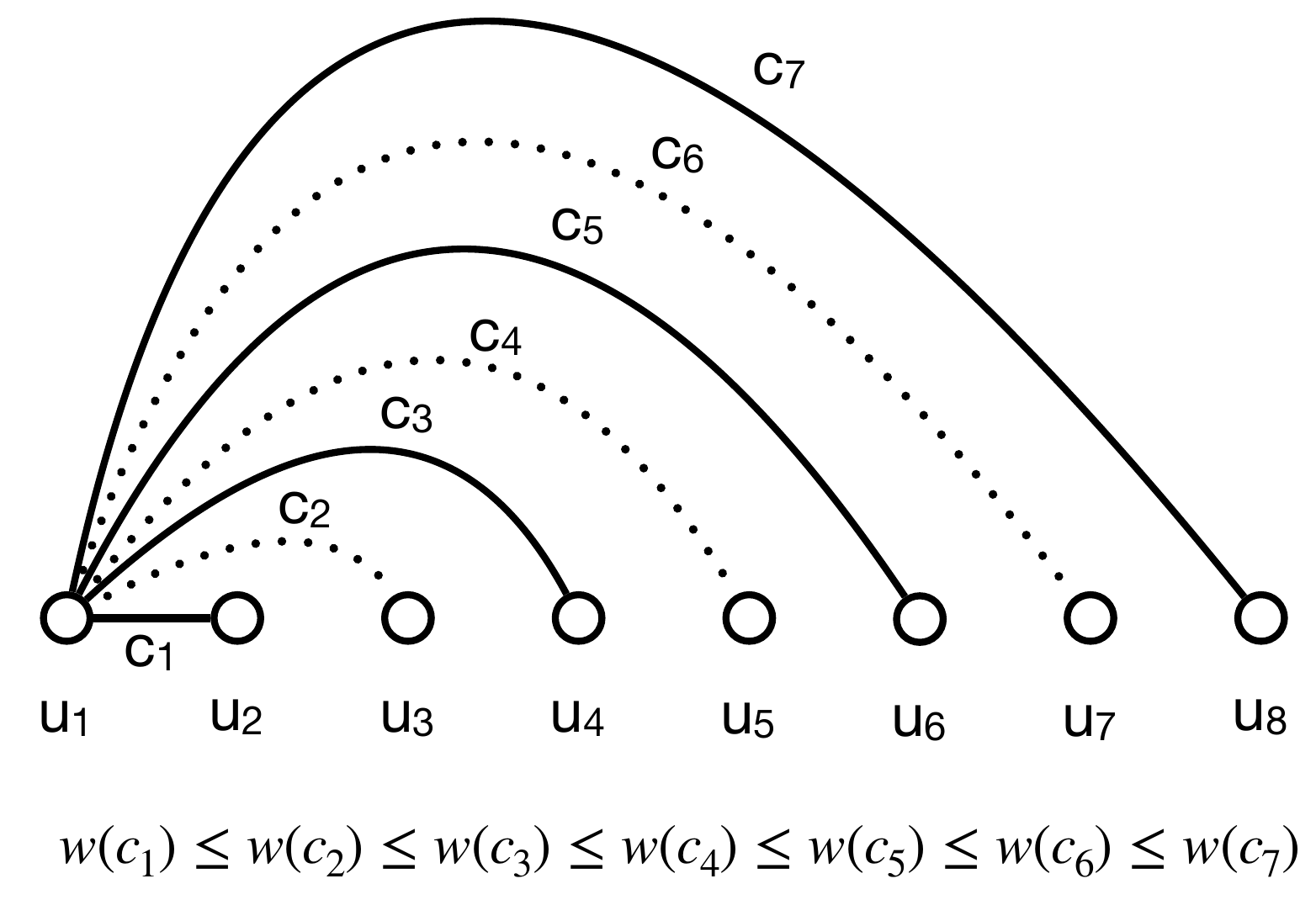}
    \caption{\label{fig:image3}}
    \end{subfigure}
\caption{(\subref{fig:image1}) An example of a $3$-FSeq  and (\subref{fig:image2}) the corresponding $3$-FP  sequence. (\subref{fig:image3}) Another example of a $3$-FP  sequence which does not lead to the identification of a $3$-FSeq. The dashed lines depict edges that are not in $G$, whereas the solid lines illustrate edges that are included in $G$.}\label{fig:Fseq_FP}
\end{figure}

\section{The star number of $n$-vertex graphs with $n \leq 7$}\label{sec:smallgraphs}

\subsection{$n$-vertex graphs with $n \leq 5$}
There are 34 non isomorphic graphs with 5 vertices \cite{oeis_nr}. Let $\mathcal{G}_5=\{G_1, G_2, \ldots, G_{34}\}$ be the set of these graphs. These graphs are depicted in Fig.~\ref{fig:all_graphs_5} based on increasing number of edges.  In this section, we prove that all graphs with at most 5 vertices have a star number equal to 1, with the exception of $G_{15}$, $G_{20}$, $G_{25}$, and $G_{27}$, which have a star number equal to 2. We start by proving the following lemmas.

\medskip
\begin{lemma}\label{lem:consecutive_neighb}
Let $G=(V,E)$ be a star-$1$-PCG and let $u, x \in V$  such that $x \in N(u) $, and for every other vertex $y \in N(u) $, it holds that $x y \not\in E $. For any $1$-witness graph $G^w$ of $G$, the neighborhood $N(u) $ is consecutive with respect to $\sigma(G^w)$.
%Sia $G$ un grafo   avente due nodi  $a$ e $z$  tali che $z \in  N(a)$ and per ogni altro nodo   $y \in N(a)$ vale $\{z,y\}\not \in E$.    Per   qualunque  versione pesata di  $G'$ di  $G$ si ha che the neighborhood $N(a)$ is \emph{consecutive} in $\sigma(G')$.
\end{lemma}
\begin{proof}
%To prove the lemma, we will demonstrate that a weighted graph  $G^w$,  where $N(a)$ is not consecutive in   $\sigma(G^w)$ contains a $1$-FP, and then the result will follow from Lemma~\ref{lem:kFP}. 
Let $G^w$ be a $1$-witness graph of $G$ and let $u \in V$ with $|N(u)| =t$ and let $y_1, y_2,\ldots y_t$ be the ordering of $N(u)$ according to the weight function $w$ (see Fig.~\ref{fig:consecutive_neigh}). Assume now that $N(u)$ is not consecutive in $\sigma(G^w)$ and thus there exists a vertex $z \in V-N(u) $ for which $w(y_1)\leq w(z)\leq w(y_t)$. 
Notice that if $z=u$ we have the $1$-FSeq sequence $x, u, y_t$ (if $w(x) < w(u)$) or $y_1, u, x$ (if $w(x) > w(u)$) and from Lemma~\ref{lem:Fseq} this contradicts the fact $G$ is a star-$1$-PCG. Hence, assume $z\neq u$ and notice that clearly, $z\neq x$. There are two cases to consider: either $w(x) \leq w(z)$ or $w(x) \geq w(z)$. We consider only the first case as the second one follows using similar arguments. Thus, we have $w(x)\leq w(z)\leq w(y_t)$.  Now, there are two cases to consider: if $w(u) \leq w(x)$ we have $w(u) \leq w(x)\leq w(z)\leq w(y_t)$, otherwise $w(u) \geq w(y_t)$ and we have $w(x) \leq w(z)\leq w(y_t)\leq w(u)$. In both cases, $u x, uz, uy_t$ will be a $1$-FP and we reach a contradiction by Lemma~\ref{lem:kFP}.
\end{proof}

\begin{figure}
    \centering
    \includegraphics[scale=0.25]{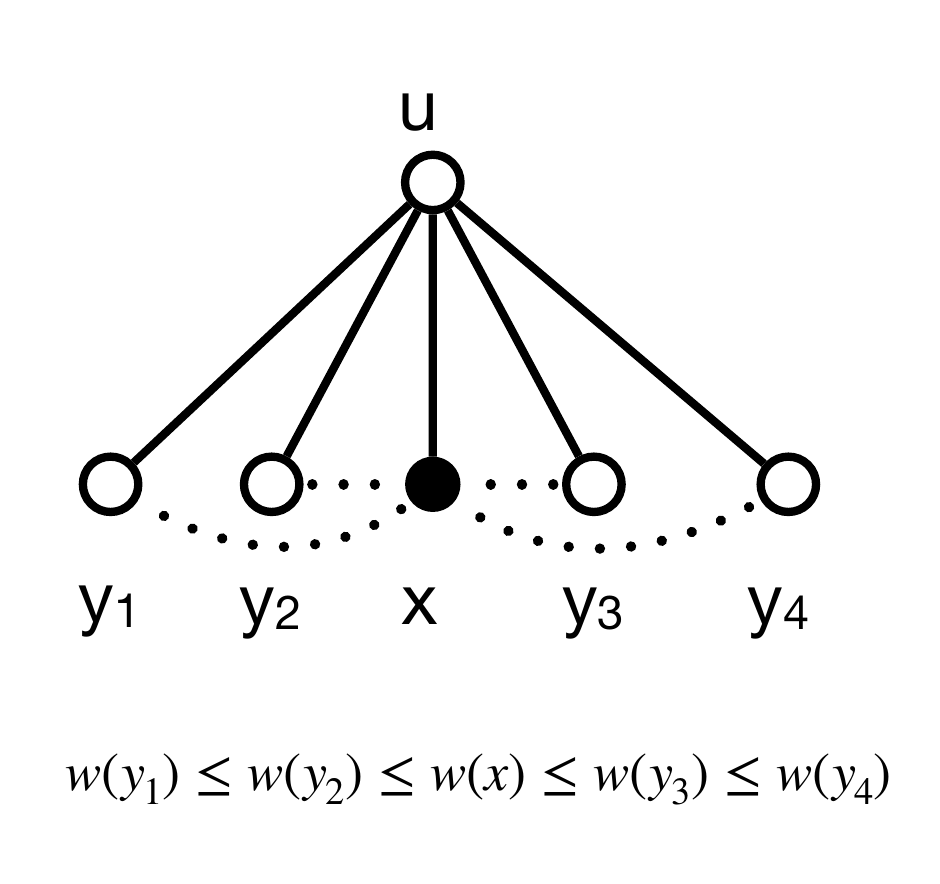}
    \caption{An illustration of Lemma~\ref{lem:consecutive_neighb} for a graph $G$. The dashed lines depict edges that are not in $G$, whereas the solid lines illustrate edges that are included in $G$. The vertices $\{y_1,y_2,x,y_3,y_4\}$ appear consecutive in $\sigma(G^w)$.}
    \label{fig:consecutive_neigh}
\end{figure}

\begin{lemma}\label{lem:g15}
$\gamma(G_{15}) \geq 2$. 
\end{lemma}
\begin{proof}
Suppose on the contrary that there exists $G^w$ which is a $1$-witness of $G_{15}$.  Let $a,b,c,d,e$ be the sequence of the vertices in the cycle (as in Fig.~\ref{fig:all_graphs_5}).  \emph{W.l.o.g.} let $a$ be the vertex of minimum weight in $G^w$ (the same argument follows for any other vertex by symmetry).  From Lemma~\ref{lem:consecutive_neighb} we have that each of the neighborhoods $N(b)=\{a,c\}$ and $N(e)=\{d,a\}$ must be consecutive in  $\sigma(G^w)$. However, there is no possible weight function $w$ for which $w(a)$ is the smallest and $a$ appears consecutive with both $c$ and $d$. Thus, we reach a contradiction and there exists no $G^w$ which is a $1$-witness of $G_{15}$.  
\end{proof}

\begin{lemma}\label{lem:g20}
$\gamma(G_{20}) \geq 2$. 
\end{lemma}
\begin{proof}
Let $a,b,c,d,e$ be the vertices of $G_{20}$ as in Fig.~\ref{fig:all_graphs_5}. Suppose on the contrary that there exists $G^w$ which is a $1$-witness of $G_{20}$. \emph{W.l.o.g.} we can assume  $w(b)<w(e)$. From Lemma~\ref{lem:consecutive_neighb} we have that each of the neighborhoods $N(b)=\{a,c,e\}$ and $N(e)=\{a,b,d\}$ must be consecutive in $\sigma(G^w)$. Thus, the total order $\sigma(G^w)$ must contain the partial order $\{b,d\},\{a\}, \{c,e\}$. We distinguish two cases:

\begin{itemize}
\item $w(b)<w(d)$. In this case we have only two possibilities: either $w(b) \leq w(d) \leq w(a) \leq w(c) \leq w(e)$ or  $w(b) \leq w(d) \leq w(a) \leq w(e) \leq w(c)$. In both cases we have that $b a, da, ae$ is a $1$-FP contradicting our initial hypothesis that $G^w$ is a $1$-witness.

\item $w(d)<w(b)$. In this case we have only two possibilities: either $w(d) \leq w(b) \leq w(a) \leq w(c) \leq w(e)$ or  $w(d) \leq w(b) \leq w(a) \leq w(e) \leq w(c)$. In the first case we have that  $bc, ac, ae$ is a $1$-FP and in the second case, $de, dc, bc$ is a $1$-FP. Hence, in both cases we contradict our initial hypothesis that $G^w$ is a $1$-witness.
\end{itemize}
\end{proof}

\begin{lemma}\label{lem:g25-g27}
$\gamma(G_{25}) \geq 2$ and  $\gamma(G_{27}) \geq 2$. 
\end{lemma}
\begin{proof}
Let $a,b,c,d,e$ be the vertices of $G_{25}$ as in Fig.~\ref{fig:all_graphs_5}. Suppose on the contrary that there exists $G^w$ which is a $1$-witness of $G_{25}$.   The next observation follows by the fact that $G_{25}-\{a\}$ does  contain neither a clique of three vertices nor an independent set of size three. 

\begin{observation}\label{obs:three_vertices}
Any set of three vertices in $G_{25}$ that does not contain $a$ will always include at least one edge and one non-edge of the graph $G_{25}$.
\end{observation}

Consider now the position of $a$ in $\sigma(G^w)$. We consider three cases:

\begin{itemize}
    \item $\{a\}, \{b,c,d,e\}$ is contained in $\sigma(G^w)$. Let $\sigma(G^w)=a,x,y,z,t$. By Observation~\ref{obs:three_vertices} we have that there is at least one non-edge among $xy$, $xz$, $yz$. \emph{W.l.o.g.} let $xz \not \in E(G_{25})$. Then, consider  $x,z,t$ and again by Observation~\ref{obs:three_vertices} there is at least one edge among $xt$ and $zt$. We reach a contradiction as if $xt \in E(G_{25})$, then $ax,xz,xt$ is an $1$-FP and if $zt \in E(G_{25})$, then $ax,xz,zt$ is an $1$-FP.

    \item $\{b,c,d,e\}, \{a\}$ is contained in $\sigma(G^w)$. Using similar argument as the previous item, let $\sigma(G^w)=x,y,z,t,a$. By Observation~\ref{obs:three_vertices} we have that there is at least one edge among $xy$, $xz$, $yz$. \emph{W.l.o.g.} let $xz \in E(G_{25})$. Then, consider  $x,z,t$ and again by Observation~\ref{obs:three_vertices} there is at least one non-edge among $xt$ and $zt$. We reach a contradiction as if $xt \in E(G_{25})$, then $xz,xt, ta$ is an $1$-FP and if $zt \in E(G_{25})$, then $xz,zt, ta$ is an $1$-FP.

\item $S_1, \{a\}, S_2$, where $S_1, S_2$ form a partition of $\{b,c,d,e\}$, is contained in $\sigma(G^w)$. We show that there exists $x \in S_1$ and $y \in S_2$ such that $xy \not \in E(G_{25})$. This concludes the proof as $x,a,y$ is an $1$-FSeq. \emph{W.l.o.g.} we may assume that $1 \leq |S_1| \leq 2$. If $S_1=\{x\}$ then for every vertex different from $a$ there is at least one vertex $y$ non adjacent  to it that necessarily belongs to $S_2$ and thus we have found our  $xy \not \in E(G_{25})$.  Otherwise, $S_1=\{x,z\}$ and $S_2=\{s,t\}$. Notice that one among $xs,xt,zs$ and $zt$ must be a non-edge as the graph induced by $G_{25}-\{a\}$ has strictly less then $4$ edges. Thus, again we found our pair $xy \not \in E(G_{25})$.
\end{itemize}

Finally, notice that in the previous proof, we did not consider whether there is an edge between $c$ and $d$ or not. Therefore, the argument for the graph $G_{27}$ follows identically as for the graph $G_{25}$.

\end{proof}

\begin{figure}
\includegraphics[scale=0.35]{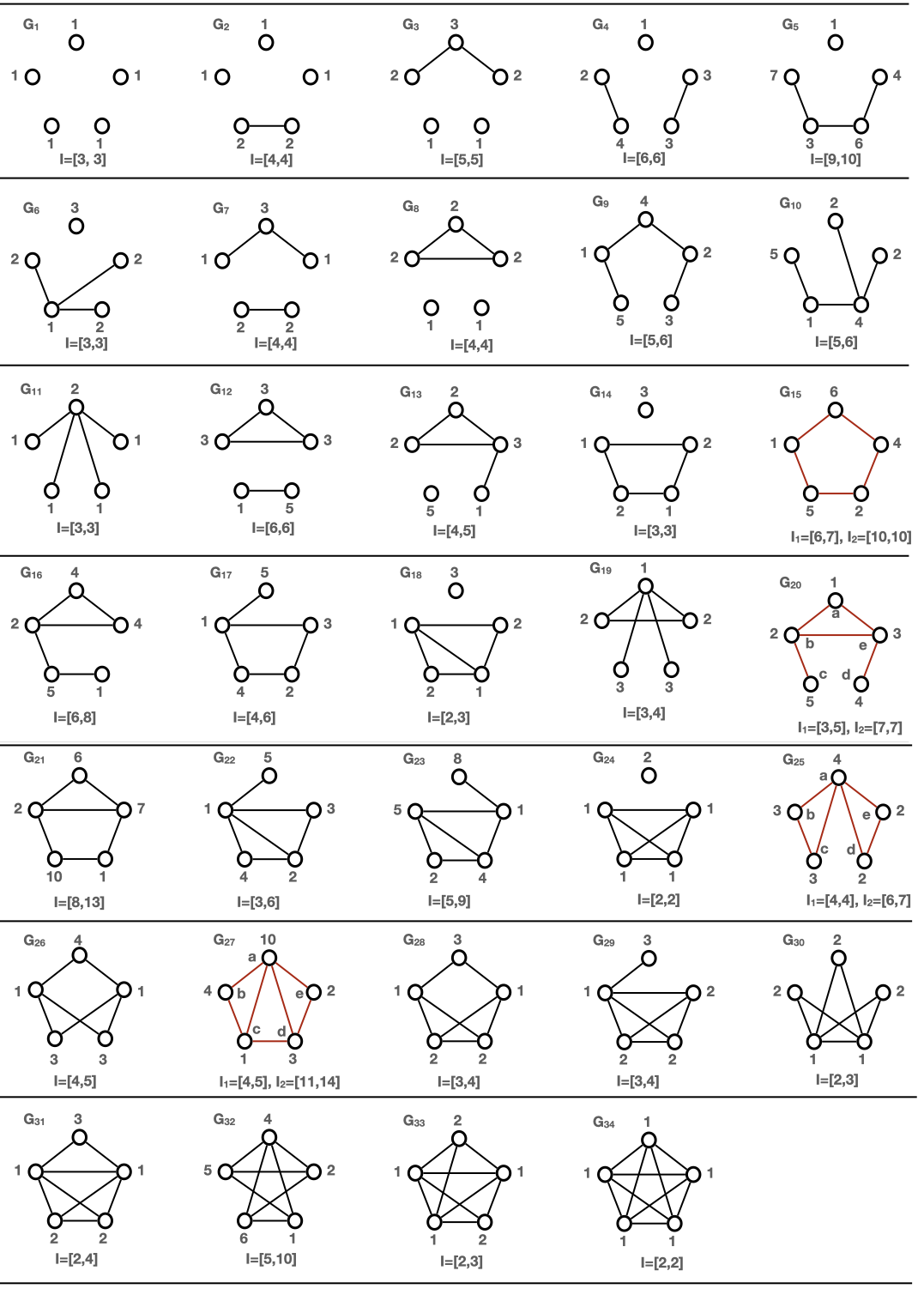}
\caption{The list for all non isomorphic graphs with at most $5$ vertices. The graphs $G_{15}, G_{20}, G_{25}, G_{27}$  are star-$2$-PCGs. The rest of the graphs are all star-$1$-PCGs. } \label{fig:all_graphs_5}
\end{figure}

\begin{theorem}\label{theo:5nodes}
For all graphs with at most 5 vertices, the star number is 1, except for the four graphs $G_{15}, G_{20}, G_{25}$ and $ G_{27}$  which have a star number equal to 2. 
\end{theorem}
\begin{proof}
From Lemma~\ref{lem:g15}, Lemma~\ref{lem:g20} and Lemma~\ref{lem:g25-g27} we conclude that $\gamma(G_i)\geq 2$ for $i \in \{15, 20, 25, 27\}$. The witness graph and the corresponding interval depicted in Fig.~\ref{fig:all_graphs_5} demonstrate that indeed $\gamma(G_i) = 2$ for these graphs. Next, it is straightforward to verify that the remaining graphs in Fig.~\ref{fig:all_graphs_5} are star-1-PCGs by examining the witness graph along with the associated interval.  This proves the claim for $n=5$.

Now, let $G'$ be a graph of at most $4$ vertices. We can obtain a graph $G$ of exactly 5 vertices, by adding isolated vertices to $G'$.  Thus, any graph of at most $4$ vertices can be viewed as an induced subgraph of the following graphs depicted in Fig.~\ref{fig:all_graphs_5}: $G_{1}, G_2 \ldots, G_{8}$, $G_{13},G_{14}, G_{18}, G_{24}$. These graphs are shown to be star-$1$-PCGs. Thus, for $n\leq 4$ the claim follows from Lemma~\ref{lem:remove_vertices}.
\end{proof}

\subsection{$n$-vertex graphs with $n=6$ or $n=7$}\label{subsec:6and7}
The number of non-isomorphic graphs with 6 and 7 vertices  is 156 and 1044, respectively  \cite{oeis_nr}, which is significantly greater then that of graphs on 5 vertices. Hence, to determine the star number of these graphs we introduce two different algorithms that enable automatic verification. First we introduce a straightforward LP program, \emph{Algorithm 1}, that takes as input a graph $G$, an integer $k$ denoting the number of intervals required and an integer $M$ denoting the maximum possible weight on the vertices of $G$. The program checks if there is a $k$-witness $G^w$ where the highest vertex weight is at most $M$.  We generated all the non-isomorphic graphs with at most 7 vertices and checked if there is a $k$-witness for $k\in \{1,2,3\}$. We set $M$ empirically to the value of 10 and 20 for $n=6$ and $n=7$, respectively.  For graphs with 6 vertices, \emph{Algorithm 1}, produced either a $1$-witness or $2$-witness graph leading to the next theorem.  
\medskip
\begin{theorem}\label{theo:6nodes}
For any graph $G$ with $6$ vertices, it holds that  $\gamma(G)\leq 2$ and  there exist graphs on $6$ vertices for which the equality holds.  
\end{theorem}
\begin{proof}
In  \cite{smallgraphspage}  we include the list of all the graphs with their respective constructions proving the membership in star-$1$-PCG or star-$2$-PCG.  
\end{proof}

We applied \emph{Algorithm 1} to all the graphs with 7 vertices. For all these graphs, except for the three shown in Figure~\ref{fig:graphs7}, namely $G_{536}$,$G_{662}$ and $G_{963}$\footnote{The index of the graphs refers to the index they have in the list in \cite{smallgraphspage}.}, the algorithm successfully generated either a $1$-witness or a $2$-witness graph. For the three graphs in Figure~\ref{fig:graphs7}, it managed to construct a $3$-witness as shown in \cite{smallgraphspage} and in the caption of the figure. The proofs that the star number for these graphs is $3$ is a case by case analysis that shows that for any graph $G_i$ with $i \in \{536, 662, 963\}$, and any weight function $w$, the graph $G_i^w$ contains a $2$-FP. These proofs rely solely on the ordering of the vertices based on their weight and not on the actual weight values. The proofs are quite lengthy and technical and thus we use automatic verification to ensure that all cases were thoroughly covered. To this purpose we develop an algorithm, \emph{Algorithm 2}, that given a graph $G$ and an integer $k$, checks if there exists an ordering $\pi$ of the vertices for which no $k$-FP appears. Notice that from Lemma~\ref{lem:kFP}, if for all possible orderings $\pi$ an $k$-FP appears, then $\gamma(G) > k$. %However, notice that while this condition is necessary, its sufficiency has not been proven. This implies that even if we find an ordering $\sigma(G^w)$ in which no $k$-FP appears, we currently lack a method to demonstrate that this results in the construction of a $k$-witness graph $G^w$. 
%We briefly explain the overall idea behind \emph{Algorithm 2}. 
Notice that all three graphs in Figure~\ref{fig:graphs7} do not contain two vertices $u,v$ for which $N(u) -\{v\}=N(v)-\{u\}$ and hence by  Lemma~\ref{lem:twins}, in any $k$-witness graph $G^w$, there are no two vertices with the same weight. Thus, we can only focus strict total orders of the vertices. For each one of the graphs in Figure~\ref{fig:graphs7} and for all possible strict total orders $\pi$, \emph{Algorithm 2} found always a $2$-FP leading to the following result. 

\begin{figure}
\centering
    \begin{subfigure}[b]{0.26\textwidth}
    \centering
    \includegraphics[width=\linewidth]{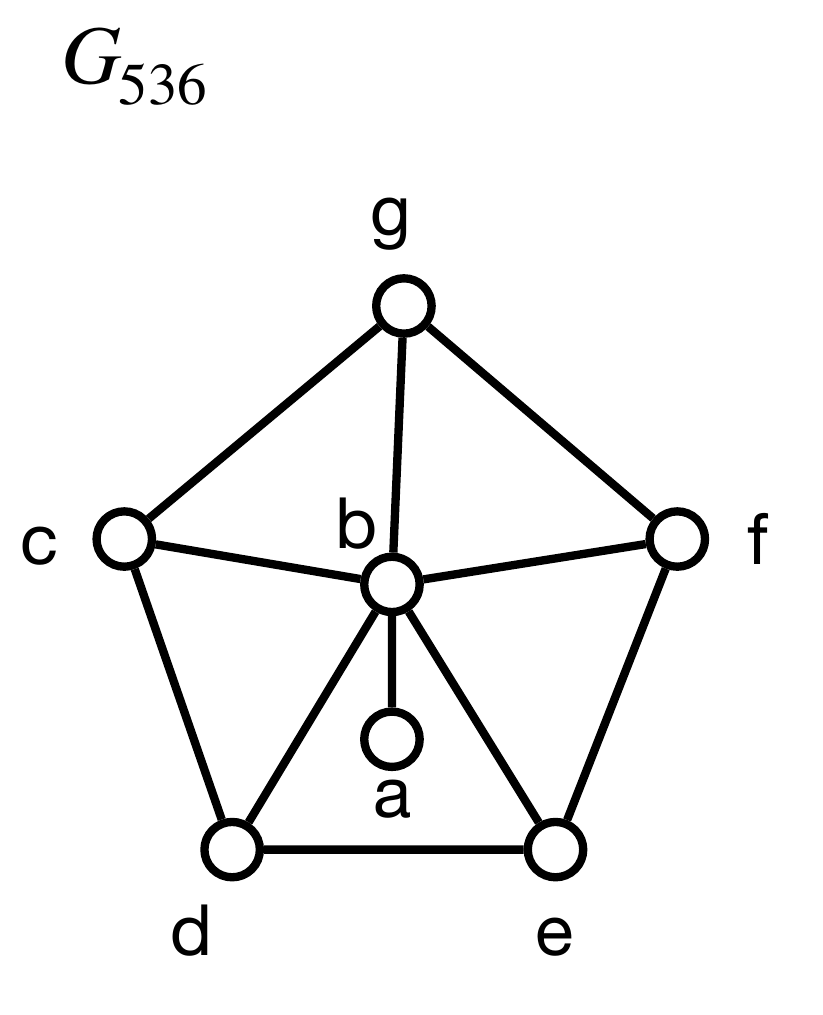}
    \caption{\label{fig:g1}}
    \end{subfigure}
    \begin{subfigure}[b]{0.3\textwidth}
    \centering
    \includegraphics[width=\linewidth]{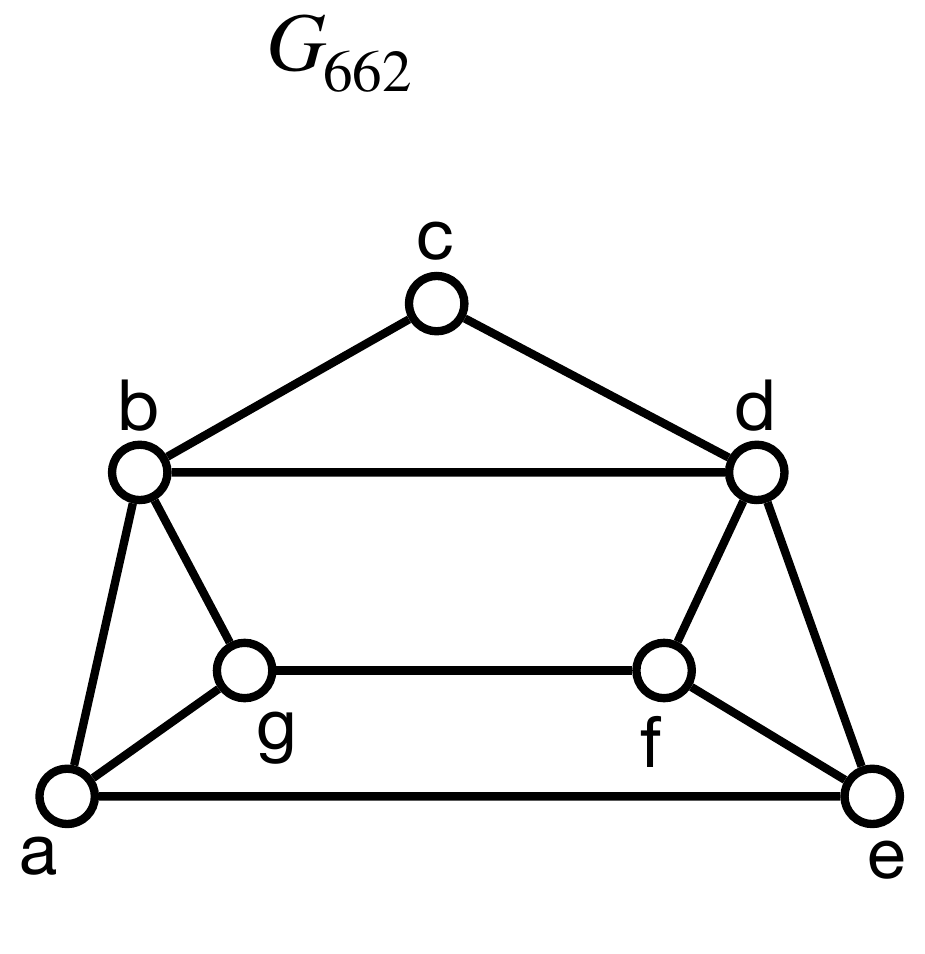}
    \caption{\label{fig:g2}}
    \end{subfigure}
    \begin{subfigure}[b]{0.28\textwidth}
    \centering
    \includegraphics[width=\linewidth]{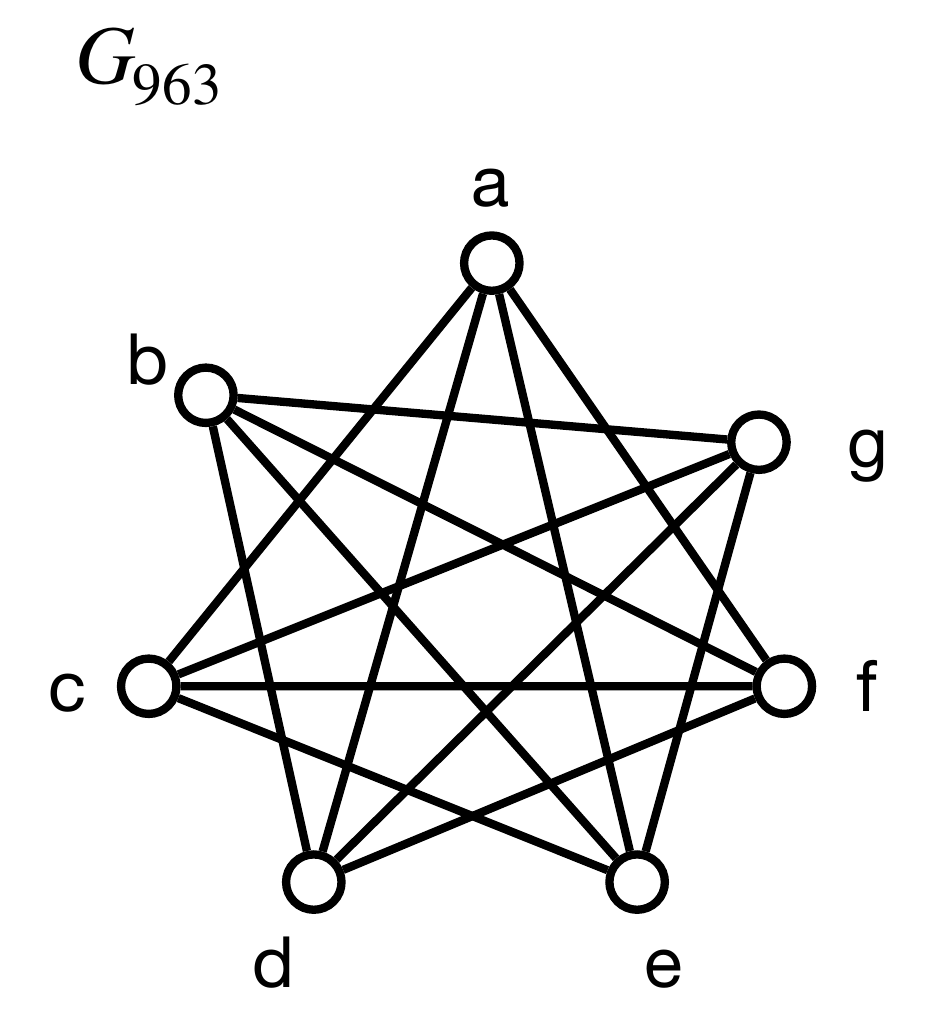}
    \caption{\label{fig:g3}}
    \end{subfigure}
\caption{The only three graphs on 7 vertices that are not star-$2$-PCG.  (\subref{fig:g1}) $\gamma(G_{536})=3$ by setting $w(a)=7, w(b)=1, w(c)=6, w(d)=4, w(e)=5, w(f)=9, w(g)=8$ and $I_1=[4,10], I_2=[14.14], I_3=[17,17]$. (\subref{fig:g2}) $\gamma(G_{662})=3$ by setting $w(a)=5, w(b)=3, w(c)=4, w(d)=1, w(e)=7, w(f)=6, w(g)=2$ and $I_1=[4,5], I_2=[7,8], I_3=[12,13]$. (\subref{fig:g3}) $\gamma(G_{963})=3$ by setting $w(a)=5, w(b)=4, w(c)=7, w(d)=2, w(e)=1, w(f)=3, w(g)=4$ and $I_1=[5,8], I_2=[10,10], I_3=[12,13]$.} \label{fig:graphs7}
\end{figure}

\medskip
\begin{theorem}\label{theo:7nodes}
For all graphs with $7$ vertices, the star number is at most $2$, except for the three graphs depicted in Figure~\ref{fig:graphs7}  which have a star number equal to 3.  
\end{theorem}
%\begin{proof}
%See the list in \cite{smallgraphspage}.
%\end{proof}

\section{The star number of caterpillars}\label{sec:caterpillar}
Here we prove that caterpillars are star-$1$-PCGs. Notice that in \cite{Jamison20} it is shown that  caterpillars are $2$-threshold and by Observation~\ref{prop:2k-threshold} they are also star-$1$-PCGs. 
However, we provide a new construction within the framework of star-$k$-PCGs, providing a different approach to understanding caterpillars in this context. %This contributes to the existing literature by presenting an alternative way to see how caterpillars fit into the spectrum of PCGs.

\medskip
\begin{theorem}
 For any caterpillar $T$, it holds $\gamma(T)=1$.
\end{theorem}
\begin{proof}
We consider first a caterpillar $T=(V,E)$ where every vertex of the central path is adjacent to \emph{exactly one} leaf. We denote by $n$ the number of vertices in the central path and enumerate its vertices $a_0, A_0, a_1, A_1$ alternatively as in Figure~\ref{fig:cat1}). Notice that we can assume $n \geq 3$ as for any $n<3$ it is easy to see that a caterpillar is a star-1-PCG \cite{smallgraphspage}.

For each $0\leq i <n$ we define the weight of $a_i$ and $A_i$  as follows:
\begin{equation}
w(a_i) = 
    \begin{cases}
    \displaystyle
        (i+1)n+1, & \text{if } i \text{ is even}\\ \notag
     (i+1)n+1+i, & \text{if } i\text{ is odd}
    \end{cases}%$}
\end{equation}

\begin{equation}
w(A_i) = 
    \begin{cases}
    \displaystyle
        n^3-in-i, & \text{if } i \text{ is even}\\ \notag
    n^3- in, & \text{if } i\text{ is odd}
    \end{cases}%$}
\end{equation}

\noindent
We define the interval: 
$$    I=[n^3+1,n^3+2n+1]
$$
\noindent
See Fig.~\ref{fig:cat2} for a construction of a $1$-witness graph for caterpillar with $n=6$.  

Notice that the only edges are of type $a_i A_i$ and $a_{2r} A_{2r+1}$ and $A_{2r+1} a_{2r+2}$ for some integer $r\geq 0$. 
Consider two vertices $x,y$ in the caterpillar. We need to consider the following cases:

\paragraph{Case I. }  $x=a_i$ and $y=a_j$.  Notice that in this case $a_ia_j \not\in E$. Indeed, $w(a_i)+w(a_j) < 2\max_r{w(a_r)} =2 w(a_{n-1}) \leq 2n^2+2n \leq  n^3$ as $n\geq 3$. Thus, $w(a_i)+w(a_j) \not\in I$. 

\paragraph{Case II. }  $x=A_i$ and $y=A_j$. Notice that in this case $A_i A_j \not\in E$. Indeed, $w(A_i)+w(A_j) > 2\min_r{w(A_r)} =2 w(A_{0}) =  2n^3 > n^3 + 2n+1$ as $n\geq 3$. Thus, $w(A_i)+w(A_j) \not\in I$. 
    
 \paragraph{Case III. }  $x=a_i$ and $y=A_j$. We have to consider the following cases:

\begin{itemize}

\item $i$ and $j$ are both odd: Notice that in this case $a_i A_j \in E$ if and only if $i=j$.  Indeed, consider 
$w(a_i)+w(A_j)=n^3+(i-j+1)n+i+1$. 

If $i=j$ then $w(a_i)+w(A_j)= n^3+n+i+1 \in I$ since $0\leq i < n$. 

If $i<j$ then $i-j\leq -2$ and thus $w(a_i)+w(A_j)\leq n^3-n+i+1 \leq  n^3 \not \in I$. 

If $i>j$ then $i-j\geq 2$ and thus $w(a_i)+w(A_j) \geq n^3 +3n +i+1 \not \in I$.

 \item $i$ and $j$ are both even: Notice that in this case $a_i A_j \in E$ if and only if $i=j$.  Indeed, consider 
$w(a_i)+w(A_j)= n^3+(i-j+1)n-j+1$. 

If $i=j$ then $w(a_i)+w(A_j)= n^3+n-j+1 \in I$ since $0\leq j < n$. 

If $i<j$ then $i-j\leq -2$ and thus $w(a_i)+w(A_j)\leq n^3-n-j+1 \leq  n^3 \not \in I$. 

If $i>j$ then $i-j\geq 2$ and thus $w(a_i)+w(A_j) \geq n^3 +3n -j+1 \geq n^3 +2n+2 \not \in I$ (where the last inequality follows as $j\leq n-1$).

\item $i$ odd and $j$ even: Notice that in this case $a_i A_j \not\in E$. Indeed, $w(a_i)+w(A_j)= 1+i+(1+i)n+ n^3-j(n+1)=n^3+(i-j+1)(n+1)$.   

If $i>j$ then $i-j \geq 1$ and we have $w(a_i)+w(A_j) \geq n^3+2(n+1) \geq  n^3+2n +2 \not \in I$. 

If $i<j$ then $i-j\leq -1$ and we have  $w(a_i)+w(A_j)\leq n^3 \not \in I$.

\item $i$ even and $j$ odd: Notice that in this case $a_i A_j \in E$ if and only if $|i-j|=1$.  Indeed, consider 
$w(a_i)+w(A_j)=n^3+(i-j+1)n+1$.

If $|i-j|=1$ we have that either $w(a_i)+w(A_j)= n^3+2n+1 \in I$ (when $i=j+1$ ) or $w(a_i)+w(A_j)=n^3+1 \in I$ (when $j=i+1$).

If $|i-j|\geq 2$ then either $w(a_i)+w(A_j) \geq  n^3 + 3n +1 \not \in I$ (when $i-j\geq 2$) or $ w(a_i)+w(A_j) \leq  n^3 -n+1 \not \in I$ (when $i-j\leq -2$). 
   
\end{itemize}

\begin{figure}
\centering
    \begin{subfigure}[t]{0.42\textwidth}
    \centering
    \includegraphics[width=\linewidth]{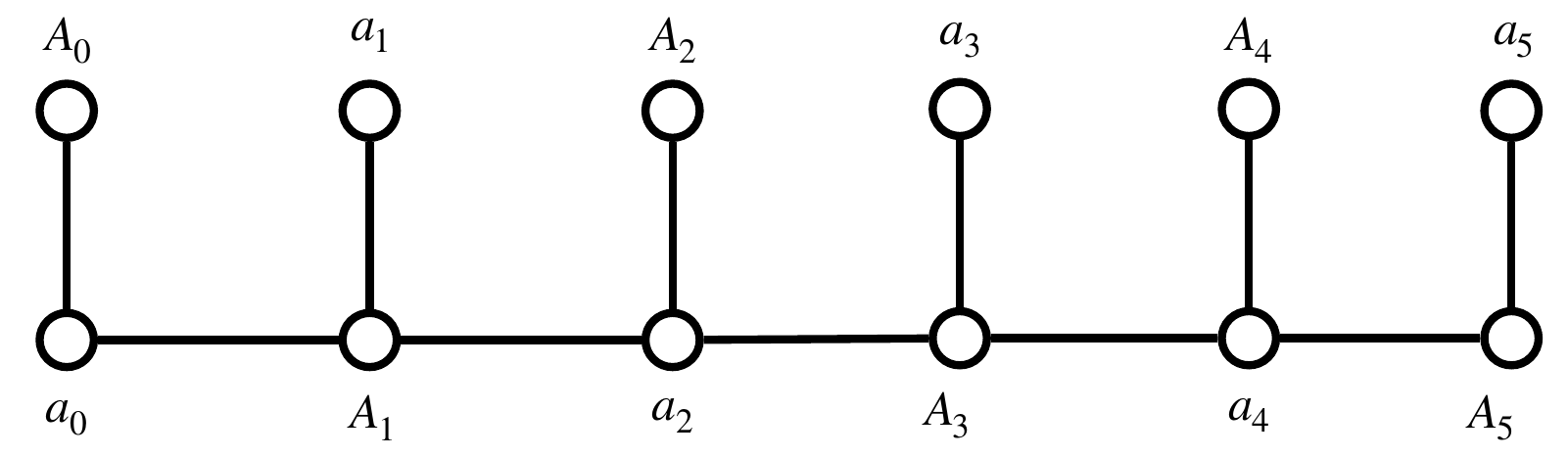}
    \caption{\label{fig:cat1}}
    \end{subfigure}
    %\quad
    \qquad 
    \begin{subfigure}[t]{0.48\textwidth}
    \centering
    \includegraphics[width=\linewidth]{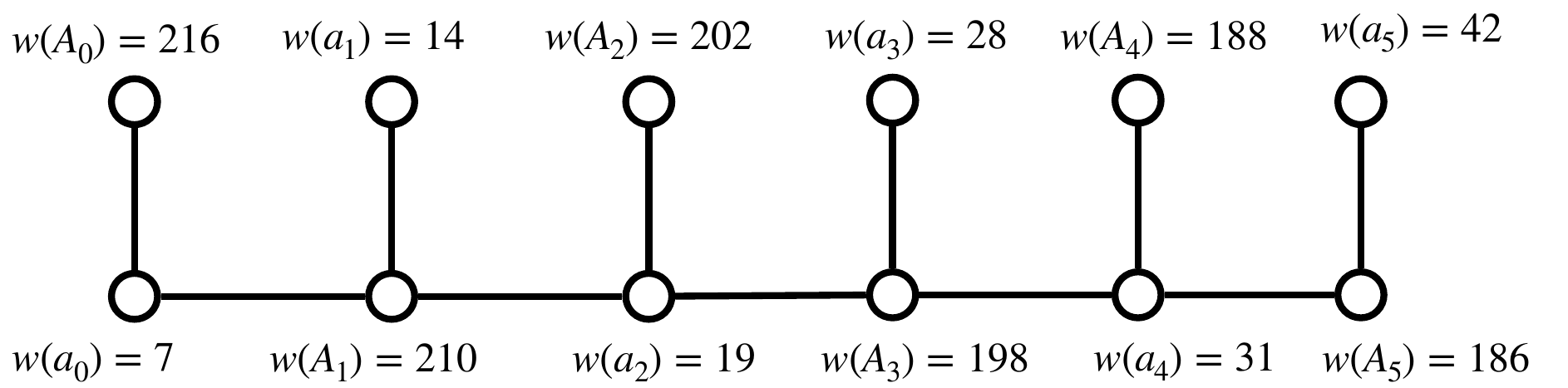}
    \caption{\label{fig:cat2}}
    \end{subfigure}
\caption{(\subref{fig:cat1}) A caterpillar with $n=6$ and (\subref{fig:cat2}) the corresponding $1$-witness graph. The interval is  $I=[217,229]$}\label{fig:fig:cat_example}
\end{figure}

\noindent
It remains to show that the construction we provided generalizes to any caterpillar. Indeed if a vertex of the spine is connected to more than one leaf, we assign the same weight to all of its leaves. The construction is still valid as from Case I and Case II we have that for any $0\leq i <n$, $2w(a_i) \not \in I$ and $2w(A_i)\not \in I$. Finally, if a vertex on the spine is not connected to any leaf, then the original construction trivially holds.  
\end{proof}

\section{The star number of cycle graphs}\label{sec:cycle} 
We already showed that $\gamma(C_5)=2$, here we show that this holds for every cycle $C_n$ with $n\geq 5$.

\medskip
\begin{theorem}\label{theo:cycle}
For any $n\geq 5$ it holds $\gamma(C_n)=2$.
\end{theorem}
\begin{proof}
Note that the proof of Lemma \ref{lem:g20} can be extended to every cycle $ C_n $ with $ n \geq 5 $ and thus $\gamma(C_n) > 1$. We provide now a construction to show that cycle graphs  are star-$2$-PCGs. To this purpose we extend a construction for the $1$-witness of a path. %Notice that it is already known that path graphs are star-$1$-PCGs \cite{Xiao2020} and two different weight functions are presented in \cite{Papan2022}.  However, for cycle graphs we need a slightly different construction.
The construction depends on the parity of $n$. Let $C_n=v_1,v_2,\ldots,v_n,v_1$ and to simplify the notation we set for every $1\leq i\leq n$, $w(v_i)=w_i$.

\paragraph{Case $n$ even. }  For every $1\leq i\leq n$, we define $w_i$ as follows: 
\begin{equation}
w_i = 
    \begin{cases}
    \displaystyle
       n+\frac{n}{2} -  \frac{i-1}{2}, & \text{if } \displaystyle i \text{ is odd}\\ \notag
    \displaystyle \frac{i}{2}, & \text{if } i \text{ is even}
    \end{cases}%$}
\end{equation}
\noindent
We define two intervals: 
%\small
\begin{align*}
    I_1=\left[n+\frac{n}{2},n+\frac{n}{2}+1\right] \quad  I_2=[2n, 2n]
\end{align*}
\noindent
In Fig.~\ref{fig:c1} we depict an example of the $2$-witness graph for $C_8$.  Consider any two arbitrary vertices $v_i, v_j$, with $v_i\neq v_j$. There are three cases to consider: 
\begin{itemize}
    \item[(a)] Both $i$ and $j$ are odd.  In this case $w_i +  w_j=3n-\frac{i+j}{2}+1>2n+1$ (where the last inequality follows from $\frac{i+j}{2}<n$). Thus, $w_i +  w_j\not\in I_1$ and $w_i +  w_j\not\in I_2$.
    \item[(b)] Both $i$ and $j$ are even. In this case $w_i +  w_j=\frac{i+j}{2} <n$. Thus, $w_i +  w_j\not \in I_1$ and $w_i +  w_j\not\in I_2$.
    \item[(c)] $i$ and $j$ have different parity. \textit{W.l.o.g.}  let $i$ be odd and $j$ be even. In this case  $w_i +  w_j=n+\frac{n}{2}-\frac{i-j-1}{2}$.  Clearly, $w_i +  w_j \in I_1$ if and only  $\frac{i-j-1}{2}\in \{-1,0\}$.  If  $\frac{i-j-1}{2} = -1$ then $i=j-1$. Otherwise if  $\frac{i-j-1}{2}=0$ we have $i=j+1$. Thus $w_i +  w_j\in I_1$  if and only if  $|i-j|=1$ which corresponds to an edge in $C_n$, more precisely in $P_n=v_1, v_2, \ldots, v_n$. Moreover, $w_i +  w_j \in I_2$ if and only if $j-i=n-1$. The latter holds  only for $j=n$ and $i=1$ as $1\leq i,j\leq n$, which again corresponds to the edge $v_n v_1 $ in $C_n$.
\end{itemize}

 \noindent
Thus, from points (a)-(c) we conclude that there is an edge among $v_i$ and $v_j$ if and only if $v_iv_j$ is an edge in $C_n$.

\paragraph{Case $n$ odd. } For every $1\leq i\leq n$ we define  $w_{i}$ so that: 
\begin{equation}
w_i = 
    \begin{cases}
    \displaystyle n-i, & \text{if } i \text{ is even}\\ \notag
    \displaystyle
       n+i, & \text{if } \displaystyle i \text{ is odd and $i\neq  n$}\\ \notag
       \displaystyle n-1, & \text{if $i=n$} 
    \end{cases}
\end{equation}
\noindent
We define two intervals: 
%\small
\begin{align*}
    I_1=[2n-1,2n+1]\quad I_2=[n,n]
\end{align*}
\noindent
In Fig.~\ref{fig:c2} we depict an example of the $2$-witness graph for $C_7$. Consider any two arbitrary vertices $v_i, v_j$, with $v_i\neq v_j$. There are three cases to consider: 
\begin{itemize}
 \item[(a)] Both $i$ and $j$ are even. In this case $w_i +  w_j=2n-(i+j) <2n-1$ (where the last inequality follows from $i+j\geq 6$). Thus $w_i +  w_j\not \in I_1$. Moreover $i+j$ is even thus $w_i +  w_j=2n-(i+j)$ is even too. Hence $w_i +  w_j\not\in I_2$.
     \item[(b)] Both $i$ and $j$ are odd. First assume $i$ and $j$ different from $n$. In this case  $w_i +  w_j=2n+(i+j)>2n+1$ (where the last inequality follows from $i+j\geq 4$). Thus $w_i +  w_j\not\in I_1\cup I_2$.
     Assume now that $i=n$, in this case we have $w_n +  w_j=n-1+n+j=2n+j-1$. If $j=1$ (\textit{i.e.} we are considering the edge $v_n v_1$) then $w_n+w_j=2n \in I_1$.
     Otherwise, for $j\geq 3$, it holds  that $w_n+w_j \geq 2n+2$. Thus,  $w_n +  w_j \not\in I_1 \cup I_2$.
    \item[(c)] The vertices $i$ and $j$ have different parity and \textit{w.l.o.g.}  let $i$ be even and $j$ be odd. Assume first that $j\neq n$. 
    In this case  $w_i +  w_j=2n-i+j$.  Clearly, $w_i +  w_j \in I_1$ if and only  $-1\leq j-i \leq 1$ and since $i\neq j$ it must be $i=j+1$ or $i=j-1$ which correspond to edges of $P_{n-1}=v_1, v_2, \ldots, v_{n-1}$.
    Consider now the case $j=n$. We have $w_i +  w_j=2n-i-1$.
    Clearly, as $i \geq 2$, $w_i +  w_j \leq 2n-3 \not \in I_1$.
    Finally,  $w_i +  w_j \in I_2$ if and only if $i=n-1$. This  corresponds to the edge $v_{n-1} v_n$ of $C_n$. 
\end{itemize}
\noindent
Thus there is an edge among $i$ and $j$ if and only if  $v_iv_j$ is an edge in $C_n$.
This concludes the proof.
\end{proof}

\begin{figure}
\centering
    \begin{subfigure}[b]{0.41\textwidth}
    \centering
    \includegraphics[width=\linewidth]{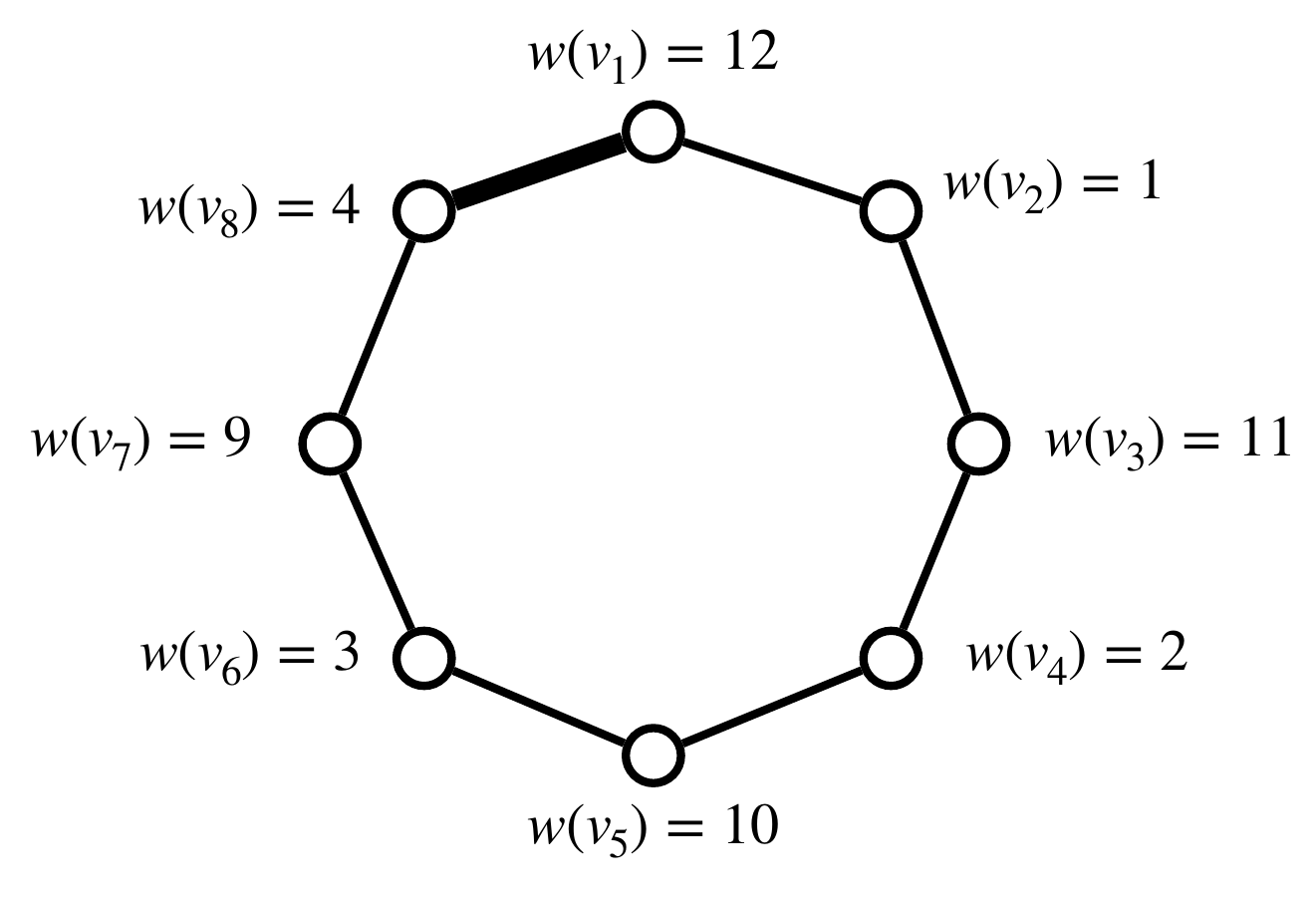}
    \caption{\label{fig:c1}}
    \end{subfigure}
    \quad \quad
    \begin{subfigure}[b]{0.4\textwidth}
    \centering
    \includegraphics[width=\linewidth]{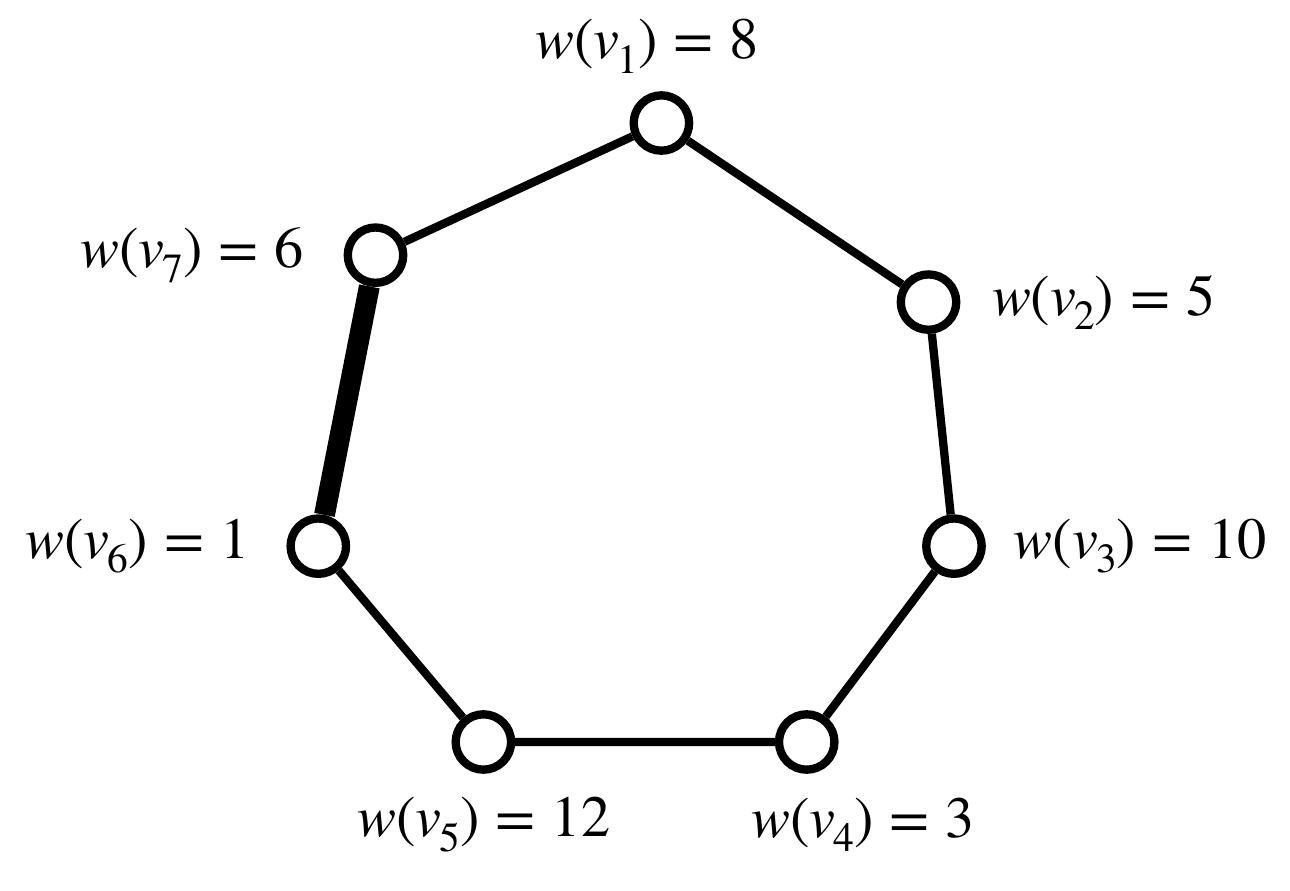}
    \caption{\label{fig:c2}}
    \end{subfigure}
\caption{(\subref{fig:c1}) The $2$-witness graph for $C_8$, the intervals are $I_1=[12,13]$ and $I_2=[16, 16]$. (\subref{fig:c2}) The $2$-witness graph for $C_7$; the two intervals are $I_1=[13,15]$ and $I_2=[7, 7]$. The normal and bold edges in the cycle correspond to edges for which the sum of the weights of their endpoints falls in $I_1$ and $I_2$, respectively.}\label{fig:cycle}
\end{figure}

\section{The star number of grid graphs}\label{sec:grids}
In \cite{Jamison20} it has been shown that every graph with minimum degree $\delta$ in which any two distinct vertices have at most $c$ neighbors in common is not a $[2(\delta-c-1)-1]$-threshold graph. Since any $d$-dimensional grid has minimum degree  $\delta = d$ and any two vertices can share at most $c=2$ neighbors by Observation~\ref{prop:2k-threshold} we have that the star number of a $d$-dimensional grid is at least $\delta-3$. Thus, the following theorem holds.
\medskip
\begin{theorem}[\cite{Jamison20}]
Given an integer $d\geq 5$, for any $d$-dimensional grid $G$ it holds $\gamma(G) \geq d-3$.   % not a star-$(d-4)$-PCG.
\end{theorem}
\medskip
\noindent
In this section we analyse the cases $d=2$ and $d=4$. To simplify the notation without overburdening it, we introduce the following convention: for any vertex $u$ described by its coordinates $(i_1, \ldots, i_d)$, we will write $w(i_1, \ldots, i_d) = w(u)$.

\medskip
\begin{theorem}
For any $2$-dimensional grid $G_{n_1,n_2}$ with $\min\{n_1, n_2\} \leq 2$ it holds $\gamma(G_{n_1,n_2})=1$.
\end{theorem}
\begin{proof}
If $\min\{n_1, n_2\} =1$ then the graph is a path and it is already known that it is a star-$1$-PCG \cite{Xiao2020,Papan2022}. Assume now  $\min\{n_1, n_2\} =2$ and \textit{w.l.o.g.} let $n_2=2$. It is worth mentioning that it is already known that $G_{n_1,2}$  is a PCG but the witness tree is a caterpillar \cite{Sammi2013}.  We define a $1$-witness graph for $G_{n_1,2}$ as follows:  For each vertex $(i,j)$ with $0\leq i \leq n_1-1$ and $0 \leq j \leq 1$, we define its weight  $w(i,j)$, as follows:
\begin{equation}
w(i,j) = 
    \begin{cases}
    \displaystyle
       2n_1-i & \text{if } i+j \text{ is even}\\ \notag
       i+1 & \text{if } i +j \text{ is odd }
    \end{cases}%$}
\end{equation}
\noindent
We define 
$$
I_1=[2n_1,2n_1+2]
$$
See Fig.~\ref{fig:gr4x2} for a construction of a $1$-witness graph for $G_{2,4}$.  
We now show that $G^w$ is a $1$-witness graph for $G_{n_1,2}$. For this let  $(i,j)$ and $(i',j')$ be two vertices and we consider the following three cases:

    \paragraph{Case $j=j'=0$.} Notice that we must have $i\neq i'$ and $i+i'\leq n_1-1+n_1-3= 2n_1-4$. We consider the following three subcases:
    \begin{itemize}
        \item \textit{Both $i$ and $i'$ are odd}.  We have $w(i,0)+w(i',0)=i+i'+2\leq 2n_1-2 \not \in I_1$.
        
        \item \textit{Both $i$ and $i'$ are even}. We have $w(i,0)+w(i',0)=4n_1-(i+i')\geq 2n_1+4 \not\in I_1$.
        
        \item \textit{$i$ and $i'$ have different parity}. \textit{W.l.o.g.} assume $i$ odd and $i'$ even. Then $w(i,0)+w(i',0)=2n_1+i-i'+1$.Thus, as $I_1=[2n_1,2n_1+2]$ it must hold that either $2n_1+i-i'+1=2n_1$ or $2n_1+i-i'+1=2n_1+2$. hus, $w(i,1)+w(i',1)\in I_1$ if and only if $|i-i'|=1$, which  corresponds to the edges of $G_{n_1,2}$ for which  
        $|i-i'|=1$ and $j=j'=0$.
        \end{itemize}    
\paragraph{Case $j=j'=1$.} This case is symmetrical to the previous one. Thus, $w(i,1)+w(i',1)\in I_1$ if and only if $|i-i'|=1$, which  corresponds to the edges of $G_{n_1,2}$ for which  
        $|i-i'|=1$ and $j=j'=1$.
         
\paragraph{Case $j$ and $j'$ of different parity.} \textit{W.l.o.g.} assume $j=0$ and $j'=1$. Then we consider the following three subcases:
        \begin{itemize}
        \item \textit{Both $i$ and $i'$ are odd}. We have  $w(i,0)+w(i',1)=2n_1+i-i'+1$. If $i\neq i'$ then $|i-i'|\geq 2$ and thus either  $w(i,0)+w(i',1) \geq 2n_1+3 \not\in I_1$ or $w(i,0)+w(i',1) \leq 2n_1-1 \not\in I_1$. Otherwise, if $i=i'$ then $w(i,0)+w(i',1)=2n_1+1 \in I_1$ which corresponds to the edges of $G_{n_1,2}$ for which  
        $i=i'$ and $|j-j'|=1$.  
        
        \item \textit{Both $i$ and $i'$ are even}. We have $w(i,0)+w(i',1)=2n_1-i+i'+1$ and the case follows identical to the previous one.  Thus, $w(i,1)+w(i',1)\in I_1$ if and only if $i=i'$, which  corresponds to the edges of $G_{n_1,2}$ for which  
        $i=i'$ and $|j-j'|=1$.  
        
        \item \textit{$i$ and $i'$ have different parity}. Notice that we must have $i\neq i'$ and thus $i+i'\leq 2n_1-4$. Assume first $i$ odd and $i'$ even. Then $w(i,0)+w(i',1)=i+i'+2\leq 2n_1-2 \not\in I_1$. Otherwise, if  $i$ even and $i'$ odd. Then $w(i,0)+w(i',1)=4n_1-(i+i')\geq 2n_1+4\not\in I_1$.
        \end{itemize}    
\noindent
Thus, $G_{n_1,2}^w$ is a $1$-witness graph for $G_{n_1,2}$. \end{proof}

\begin{figure}
\centering
    \begin{subfigure}[t]{0.38\textwidth}
    \centering
    \includegraphics[width=\linewidth]{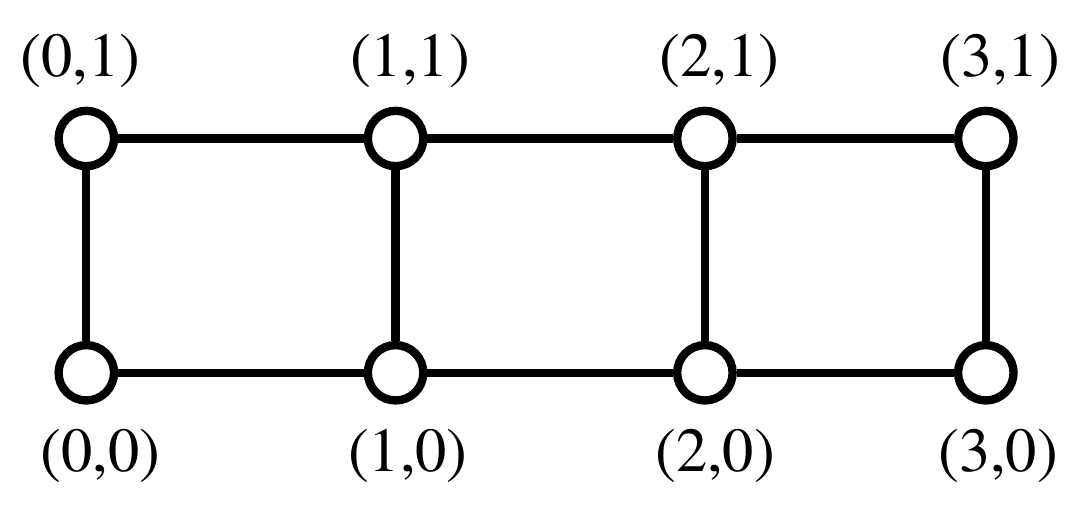}
    \caption{\label{fig:grid1}}
    \vspace{0.15cm}
    \end{subfigure}
    \quad
    \begin{subfigure}[t]{0.48\textwidth}
    \centering
    \includegraphics[width=\linewidth]{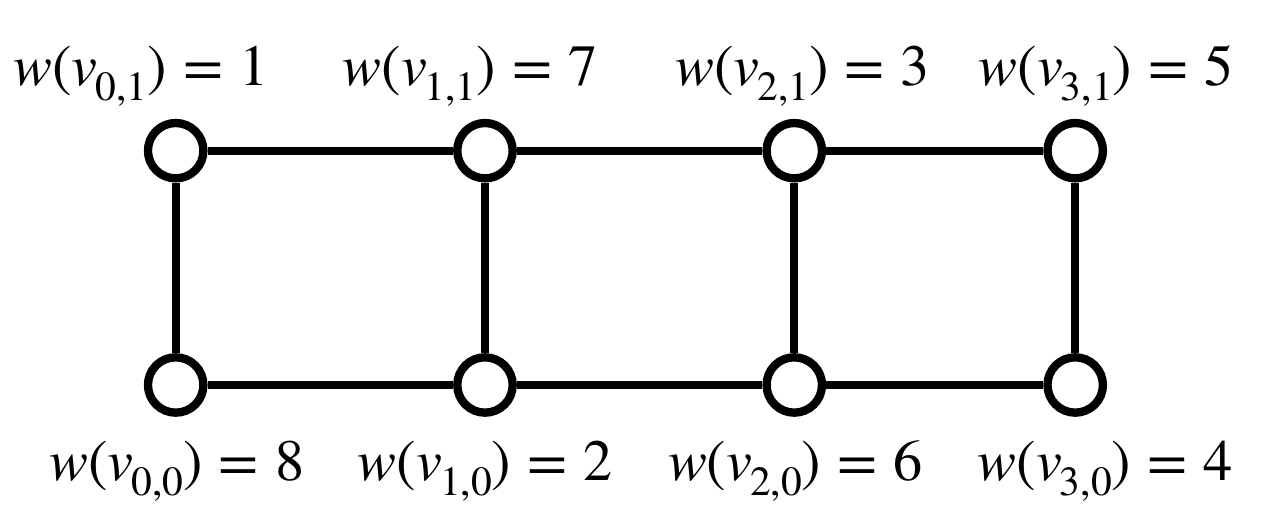}
    \caption{\label{fig:grid2}}
    \end{subfigure}
\caption{(\subref{fig:grid1}) The grid $G_{4,2}$. (\subref{fig:grid2}) The $1$-witness graph for $G_{4,2}$ with interval $I_1=[8,10]$.}\label{fig:gr4x2}
\end{figure}

\medskip
\begin{theorem}
For any $2$-dimensional grid $G_{n_1,n_2}$ with $n_1, n_2\geq 3$ it holds \mbox{$\gamma(G_{n_1,n_2})=2$.}
\end{theorem}
\begin{proof}
Notice that the cycle graph $C_8$ is an induced subgraph of $G_{n_1,n_2}$ with $n_1,n_2\geq 3$. By Lemma~\ref{lem:remove_vertices} and Theorem~\ref{theo:cycle}  we have $\gamma(G_{n_1,n_2})\geq 2$.
We now prove that any $2$-dimensional grid graph is a star-$2$-PCG. Notice that if a graph $G$ is a star-$2$-PCG, then so is any vertex induced subgraph of $G$. Hence, it is sufficient to focus on the case where $n_1=n_2=h$. Indeed the construction for any $G_{n_1,n_2}$ can be obtained by the one of $G_{h,h}$ where $h=\max \{n_1, n_2\}$. Let $G=G_{h,h}$ and for each vertex $(i,j)$, we define its weight $w(i,j)$, as follows:
\begin{equation}
%\resizebox{0.47\textwidth}{!}{$
w(i,j) = 
    \begin{cases}
    \displaystyle
       \frac{\displaystyle (i+j-1)h}{\displaystyle 2} {\displaystyle + i +1}, & \text{if } i+j \text{ is odd}\\ \notag
     (2h-1)h - \frac{\displaystyle (i+j)h}{\displaystyle 2} {\displaystyle - i}, & \text{if } i+j \text{ is even}
    \end{cases}%$}
\end{equation}
\noindent
We define two intervals: 
%\small
\begin{align*}
    I_1&=[2h(h-1),2h(h-1)+1]\\ \notag
    I_2&=[2h(h-1)+h+1,2h(h-1)+h+2]
\end{align*}
%\normalsize
\noindent
See Fig.~\ref{fig:grid_constr} for a construction of a $2$-witness graph for $G_{4,4}$.  By construction two vertices $(i,j)$ and $(i',j')$  are adjacent if the sum of their weights is one of the $4$ integer values in $ S=\{2h(h-1), 2h(h-1)+1, 2h(h-1)+h+1, 2h(h-1)+h+2\}$. Consider two vertices $(i,j)$, $(i',j')$ in the grid with $0\leq i,j,i',j' \leq h-1$. There are two cases to consider. 
\noindent
\paragraph{Case $i+j$  and $i'+j'$  have the same parity.} By definition of a $2$-dimensional grid these vertices are not adjacent.  Consider first the case where  $i+j$ and $i'+j'$ are both odd. Notice that $ w(i,j)=\frac{(i+j-1)h}{2} + i +1$  is maximized for $i=h-1$ and $j=h-2$ (notice that as $i+j$ is odd we cannot have $i=j=h-1$. Hence, $ w(i,j)+w(i',j') < (2h-4)h+2h=(2h-2)h$ and thus is not in $S$ where the last inequality holds as $(i,j)\neq (i',j')$. 

Consider now the case $i+j$ and $i'+j'$ are both even. Notice that $ w(i,j)=(2h-1)h - \frac{\displaystyle (i+j)h}{\displaystyle 2} {\displaystyle - i}$ is minimized for $i=h-1$ and $j=h-1$. Hence, $ w(i,j)+w(i',j') > 2h(h-1)+2$  and thus is not in $S$. 
\noindent
\paragraph{Case $i+j$  and $i'+j'$  have the different parity.} \textit{W.l.o.g.} assume $i+j$ odd and $i'+j'$ even and thus: 
\normalsize
$$
w(i,j)+w(i',j') = (2h-1)h + (i-i'+j-j'-1)\frac{h}{2} + i - i' +1  
$$
\noindent
We consider now for what values of $i,j,i',j'$ we have $w(i,j)+w(i',j')=s \in S$. To this purpose we  solve the following equations for each possible value of $s$.
\begin{itemize}
    \item In the case  $s=2h(h-1)$  we obtain the equation $(i-i'+j-j'+1)h+2(i - i'+1)=0$. Let $c=i-i'+j-j'+1$ and we consider for which values of $c$ the equation has solutions. Notice first that for $c\leq -3$ and $c\geq 3$ there are no solutions as  $-2h\leq 2(i - i'+1)\leq 2h $ (recall that $0\leq i,j,i',j'\leq h-1$). Moreover, as $i+j$ is odd and $i'+j'$ is even we have that $c$ must be even. Thus, the only possible cases that remain to consider are $c\in \{-2,0,2\}$.  If $c=2$ then $2h+2(i-i'+1)=0$ and thus $i-i'+1=-h$. From this  and $i-i'+j-j'+1=2$ we have $j-j'=h+2$ which is not possible since $j-j'\leq h-1$. If $c=-2$ then $-2h+2(i-i'+1)=0$ and thus $i-i'+1=h$. From this  and $i-i'+j-j'+1=-2$ we have $j-j'=-(h+2)$ which is not possible since $j-j'\geq -(h+1)$.  The only possibility is $c=0$ and as a consequence $2(i - i'+1)=0$ from which we have $i=i'-1$. Then as $i-i'+j-j'+1=0$ we have $j=j'$. 
 
 \item In the case  $s=2h(h-1)+1$ following a similar argument as in the previous point we have the only possibility is $i=i'$ and $j=j'-1$.

\item In the case  $s = 2h(h - 1) + h + 1$ following a similar argument as in the previous point we have the only possibility is $i=i'$ and $j=j'+1$.

\item In the case  $s = 2h(h - 1) + h + 2$ following a similar argument as in the previous point we have the only possibility is $i=i'+1$ and $j=j'$.
    
\end{itemize}

\noindent
From the previous four items we have that two vertices $(i,j)$, $(i',j')$ are adjacent if and only if $i=i'$ and $|j-j'|=1$ or $j=j'$ and $|i-i'|=1$ and $j=j'$. This concludes the proof.
\end{proof}

\begin{figure}
\centering
    \begin{subfigure}[t]{0.32\textwidth}
    \centering
    \includegraphics[width=\linewidth]{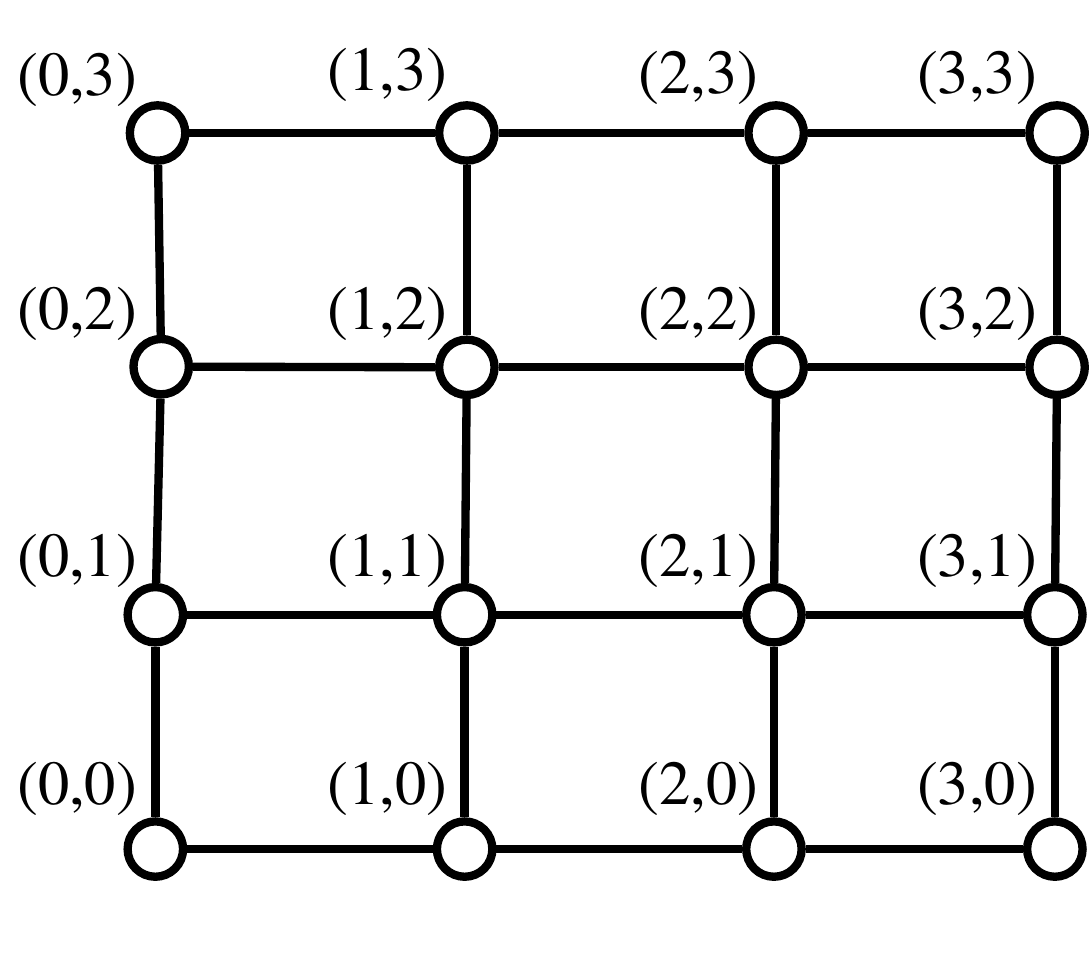}
    \caption{\label{fig:4g1}}
    \end{subfigure}
    \qquad \qquad 
    \begin{subfigure}[t]{0.32\textwidth}
    \centering
    \includegraphics[width=\linewidth]{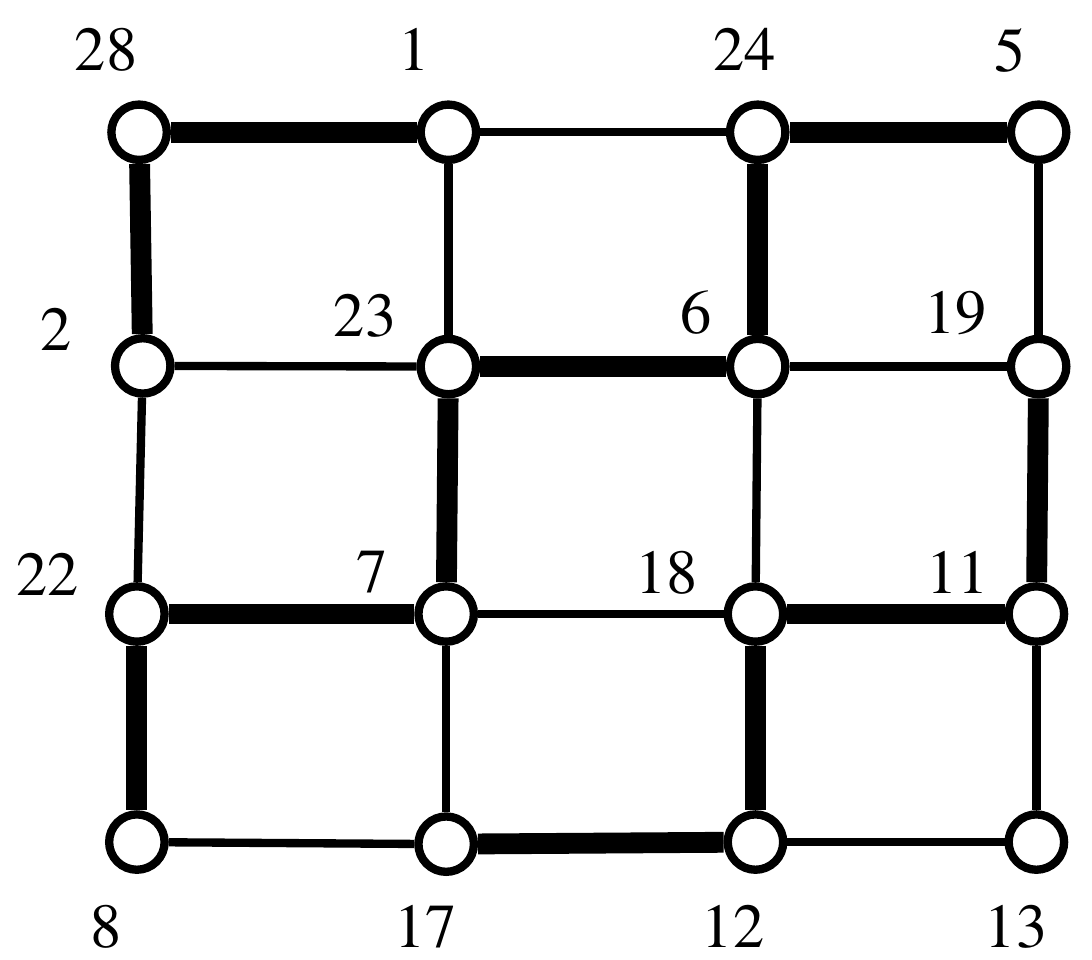}
    \caption{\label{fig:4g2}}
    \end{subfigure}
\caption{(\subref{fig:4g1}) The grid $G_{4,4}$. (\subref{fig:4g2}) The corresponding $2$-witness graph  graph.  The two intervals are $I_1=[24,25]$ and $I_2=[29,30]$. The normal and bold edges correspond to edges for which the sum of their endpoints falls in $I_1$ and $I_2$, respectively.}\label{fig:grid_constr}
\end{figure}

We now consider $4$-dimensional grids.

\medskip
\begin{theorem}
Given four positive integers $n_1, n_2, n_3,n_4 \geq 3$, for any $4$-dimensional grid $G_{n_1,n_2,n_3,n_4}$it holds $\gamma(G_{n_1,n_2,n_3,n_4}) \geq 3$.
\end{theorem}
\begin{proof}
Consider $G=G_{n_1,n_2,n_3,n_4}$, and let $G^{\omega}$ be any of its $k$-witness graphs. We will show that  in $G^{\omega}$  there exists a vertex $x$ with $v_1, v_2, v_{3}$ in $N(x)$ and  $u_1, u_2$ not in $N(x)\cup \{x\}$ such that \mbox{$w(v_1) \leq w(u_1) \leq w(v_2)\leq  w(u_2)\leq w(v_{3})$} and the proof will follow from Lemma \ref{lem:kFP} since the existence of the vertex $x$ implies the presence in $G^{\omega}$ of the  $2$-FP $x v_1, x u_1, xv_2, xu_2, xv_3$.

\noindent 

Consider the vertex  $a=(1,1,1,1)$. Notice that $a$ has exactly  $8$ neighbors and by Lemma \ref {lem:twins} all these vertices have different weights, \textit{w.l.o.g.} let $b_1, \ldots, b_8$ be these  neighbors and $w(b_1)< w(b_2)< \ldots <w(b_8)$. We consider now two cases:

\paragraph{Case I. } 
{\em There exist two vertices $u_1,u_2 \not\in N(a)$ and two integers $i$ and $j$,  $1\leq i<j<7$ such that  it holds $w(b_i)<w(u_1)<w(b_{i+1})$ and $w(b_j)<w(u_2)<w(b_{j+1})$.}
In this case we have $x=a$ and $v_1=b_1$, $v_2=b_{i+1}$, $v_{3}=b_8$.

\paragraph{Case II. }{\em  There exists an integer $i$, $1\leq i\leq 7$ such that for any $u\not\in N(a)$  one of the followings holds: (ii) $w(u) < w(b_1)$, (ii) $w(b_i) < w(u) < w(b_{i+1})$, (iii) $w(u) > w(b_8)$.} We will consider only the cases $1\leq i\leq 4$ as the reasoning in the cases for $i>4$ is identical to the cases $8-i$. 

\noindent
We  make use of the following definitions.

Two neighbors  $b$ and $b'$ of vertex $a$  are called \emph{\oppos}\  if there exists a dimension $i$ for which $|b_i - b'_i| = 2$ and for all $j\neq i$, $b_j = b'_j$.  Notice that the $8$ neighbors of  $a$ are  partitioned in $4$ pairs of \oppos\ vertices. 
Consider now  any two  neighbors $b_i$ and $b_j$ of $a$ that are not \oppos. Then  there exists exactly one vertex $y$ different from $a$, such that $N(a) \cap N(y)=\{b_i, b_j\}$.
We denote this vertex  as  $Q^a_{b_i, b_j}$.

\paragraph{Case II.a. $1\leq i\leq 2$.}  At least one between
the pairs $(b_4, b_6)$ and $(b_4, b_7)$ is not an opposed pair. \textit{W.l.o.g.} let $(b_4, b_6)$  be such a pair. Let $y = Q^a_{b_4, b_6}$ and consider an arbitrary vertex $v \in N(y)-\{b_4,b_6\}$. Since $v \not\in N(a)$ we have that one of the followings must hold: (i) $w(v) < w(b_1)$, (ii) $w(b_i) < w(v) < w(b_{i+1})$, (iii) $w(v) > w(b_8)$.  We consider each case separately and prove that in all the cases we have $x=y$. 

\begin{itemize}
    \item[(i)] If $w(v) < w(b_1)$ we can set $v_1=v$, $v_2=b_4$, $v_3=b_6$ and $u_1=b_1$, $u_2=b_5$. 
    \item[(ii)] If $w(b_i) < w(v) < w(b_{i+1})$ we can set $v_1=v$, $v_2=b_4$, $v_3=b_6$ and $u_1=b_3$, $u_2=b_5$. 
    \item[(iii)] If $w(v) > w(b_8)$ we can set $v_1=b_4$, $v_2=b_6$, $v_3=v$ and $u_1=b_5$, $u_2=b_8$. 
\end{itemize}

\paragraph{Case II.b. $3\leq i\leq 4$. } At least one between
the pairs $(b_2, b_6)$ and $(b_2, b_7)$ is not an opposed pair. \textit{W.l.o.g.} let $(b_2, b_6)$  be such a pair. Let $y = Q^a_{b_2, b_6}$,  similarly to the previous case let  $v \in N(y)-\{b_2,b_6\}$.  Since $v \not\in N(a)$ we have that one of the followings must hold: (i) $w(v) < w(b_1)$, (ii) $w(b_i) < w(v) < w(b_{i+1})$, (iii) $w(v) > w(b_8)$.  We consider each case separately and prove that in all the cases we have $x=y$.

\begin{itemize}
    \item[(i)] If $w(v) < w(b_1)$ we can set $v_1=v$, $v_2=b_2$, $v_3=b_6$ and $u_1=b_1$, $u_2=b_5$. 
    \item[(ii)] If $w(b_i) < w(v) < w(b_{i+1})$ we can set $v_1=b_2$, $v_2=v$, $v_3=b_6$ and $u_1=b_3$, $u_2=b_5$. 
    \item[(iii)] If $w(v) > w(b_8)$ we can set $v_1=b_2$, $v_2=b_6$, $v_3=v$ and $u_1=b_5$, $u_2=b_8$. 
\end{itemize}
This concludes the proof.
\end{proof}

\section{Conclusions and open problems}
In this paper we consider the problem of characterizing  star-$k$-PCGs. This is particularly interesting as this class connects two important graph classes: the PCGs and multithreshold graphs for both of which a complete characterization is not yet known.  Here we investigate the star number, $\gamma$, of simple graph classes, such as graphs of small size, caterpillars, cycles and grids.  Specifically, we determine the exact value of $\gamma(G)$ for all the graphs with at most 7 vertices. By doing so we  show that the smallest graphs with star number 2 are only 4 and have exactly 5 vertices; the smallest graphs with star number 3 are only 3 and have exactly 7 vertices. Next, we provide a construction showing that the star number of caterpillars is one. Moreover, we show that the star number of cycles and two dimensional grid graphs is 2 and that for  $4$-dimensional grids the star number is at least 3. Many  problems remain open.

\medskip
\begin{problem}
For a $d$-dimensional grid $G$, with $d\neq 2$, determine the value of $\gamma(G)$.
\end{problem}
\medskip

Given a graph $G$, from Lemma~\ref{lem:kFP} it follows that if for any ordering  $\pi$ of the vertices, the presence of a $k$-FP can be deduced, then $\gamma(G) > k$. However, even if we find an ordering $\pi$ in which no $k$-FP appears, we currently lack a method to obtain a $k$-witness graph $G^w$. Thus we have
\medskip
\begin{problem}
Given a graph $G$, is it true that the existence of an ordering  of the vertices $\pi$ such that no $k$-FP can be deduced, implies the existence of a weight function $w$, for which $G^w$ is a $k$-witness graph and $\sigma(G^w)=\pi$?
\end{problem}
\medskip

For $k=1$ a similar result has been proved in \cite{Xiao2020} so a positive answer to the previous problem can be seen as a generalization of the result in  \cite{Xiao2020} to an arbitrary $k$. Notice that this is can be interesting as in \cite{Xiao2020} it lead to a polynomial time algorithm for the recognition of star-$1$-PCGs.  Thus, this may be an important step toward the solution of the following problem.

\medskip
\begin{problem}
Determine the computational complexity of recognizing star-$k$-PCGs for $k\geq 2$.
\end{problem}
\medskip

%Notice that the previous problem is open even in the case $k\geq 2$, when $k$ is not part of the input, and also 

From the results of Section~\ref{sec:smallgraphs} we have that for all graphs with at most 7 vertices,  the graph in Fig.~\ref{fig:tree_notPCG} is the only acyclic graph that is not a star-$1$-PCG. To the best of our knowledge, there are no acyclic graphs with a star number of at least 3 documented in the  literature. 

\begin{figure}
\centering
    \includegraphics[scale=0.2]{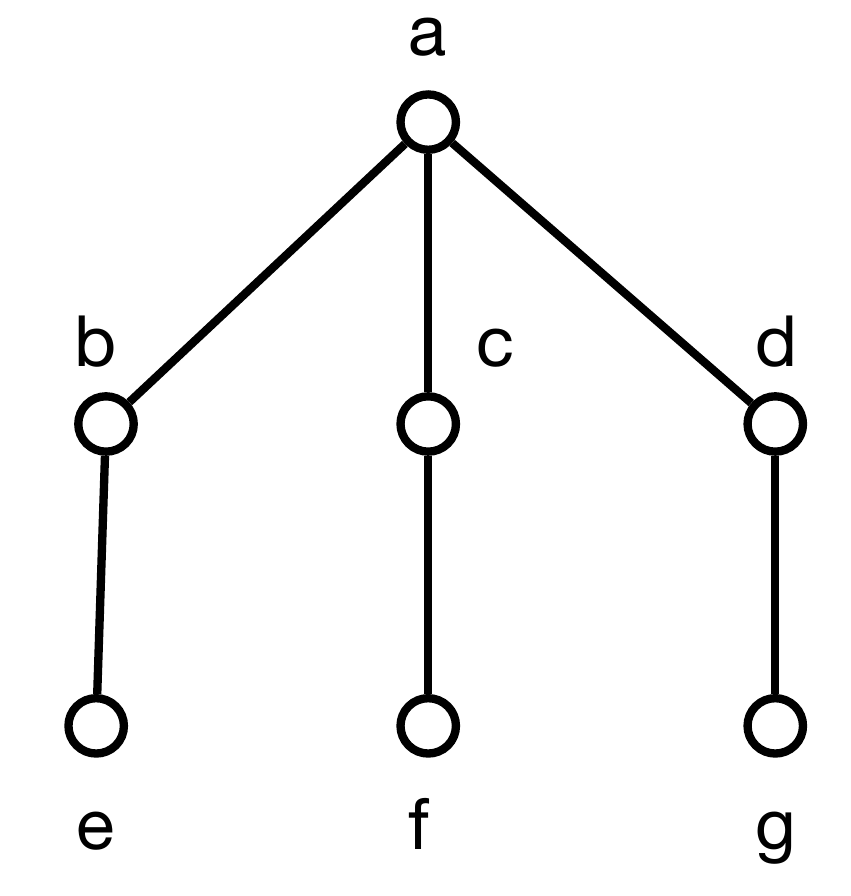}
    \caption{The smallest tree known not to be a star-$1$-PCG.}\label{fig:tree_notPCG}
\end{figure}

\medskip
\begin{problem}
Is there an acyclic graph $G$ for which $\gamma(G) > 2$?   
\end{problem}

\noindent
More generally,
\medskip
\begin{problem}
 What is the star number of acyclic graphs?   
\end{problem}
\medskip

\bibliographystyle{plain}
\bibliography{my-references}% 

\end{document}